\documentclass[12pt, oneside]{amsart}   	
\usepackage[margin= 1in]{geometry}                		
\usepackage{graphicx}				
\usepackage{amssymb}
\usepackage{mathrsfs}
\usepackage{amsthm}
\usepackage{amsmath}
\usepackage{hyperref}
\usepackage{tikz-cd}
\usepackage{dsfont}
\usepackage{tikz}
\usepackage{tikz-3dplot}
\usetikzlibrary{math}
\newcounter{thm}
\newtheorem{remark}[thm]{Remark}
\newtheorem{proposition}[thm]{Proposition}
\newtheorem{definition}[thm]{Definition}
\newtheorem{theorem}[thm]{Theorem}

\newtheorem{lemma}[thm]{Lemma}
\newtheorem{corollary}[thm]{Corollary}

\newcommand{\NN}{\mathbb{N}}
\newcommand{\ZZ}{\mathbb{Z}}
\newcommand{\kk}{\mathds{k}}
\newcommand{\RR}{\mathbb{R}}
\newcommand{\val}{\operatorname{val}}
\newcommand{\vir}{\operatorname{vir}}

\newcommand{\pt}{\operatorname{pt}}

\newcommand{\ev}{\operatorname{ev}}

\newcommand{\Spec}{\operatorname{Spec}}
\newcommand{\gp}{\operatorname{gp}}

\newcommand{\Gr}{\operatorname{Gr}}
\newcommand{\tor}{\operatorname{tor}}
\newcommand{\trop}{\operatorname{trop}}
\newcommand{\out}{\operatorname{out}}
\newcommand{\Pic}{\operatorname{Pic}}
\newcommand{\an}{\operatorname{an}}
\newcommand{\lcm}{\operatorname{lcm}}
\newcommand{\Forget}{\operatorname{Forget}}
\newcommand{\NE}{\operatorname{NE}}
\newcommand{\coker}{\operatorname{coker}}
\makeatletter
\@namedef{subjclassname@2020}{\textup{2020} Mathematics Subject Classification}
\makeatother

\usetikzlibrary{fadings}

\numberwithin{equation}{section}
\numberwithin{thm}{section}
\numberwithin{figure}{section}

\title{Quantum periods, toric degenerations and intrinsic mirror symmetry}
\author{Samuel Johnston}
\address{Samuel Johnston, Imperial College London, South Kensington Campus, London SW7 2AZ, UK}
\email{samuel.johnston@imperial.ac.uk}
\date{}		
\subjclass[2020]{14J33,14N35,14N10}

\date{}
\begin{document}

\begin{abstract}
Given a Fano variety $X$, and $U$ an affine log Calabi-Yau variety given as the complement of an anticanonical divisor $D \subset X$, we prove that for any snc compactification $Y$ of $U$ dominating $X$ with $D' = Y\setminus U$, there exists an element $W_D \in R_{(Y,D')}$ of the intrinsic mirror algebra whose classical periods give the regularized quantum periods of $X$. Using this result, we deduce various corollaries regarding Fano mirror symmetry, in particular integrality of regularized quantum periods in large generality and the existence of Laurent mirrors to all Fano varieties whose mirrors contain a dense torus. When $U$ is an affine cluster variety satisfying the Fock-Goncharov conjecture, we use this result to produce a family of polytopes indexed by seeds of $U$ determined by enumerative invariants of  the pair $(X,D)$ which give a family of Newton-Okounkov bodies and toric degenerations of $X$. Moreover, we give an explicit description of the superpotential in the Grassmanian setting, in particular recovering the Pl\"ucker coordinate mirror discovered by Marsh and Rietsch. Finally, we use the main result to show that the quantum period sequence is equivalent to all theta function structure constants for $R_{(X,D)}$ when $D$ is a smooth anticanonical divisor.
\end{abstract}
\maketitle

\section{Introduction}
A strong deformation invariant of a Fano variety $X$ is its quantum period sequence, a certain generating series of Gromov-Witten invariants of $X$. A prediction of Fano/LG mirror symmetry is that these quantum periods are given by oscillatory integrals of a certain function on the mirror LG model. In this paper, we show how intrinsic mirror symmetry developed by Gross and Siebert yields a natural choice of mirror Landau-Ginzburg models for a pairs $(X,D)$ of a Fano variety $X$ and an anticanonical divisor $D$, including many situations in which $(X,D)$ does not immediately have an associated intrinsic mirror algebra or canonical wall structure. 

More precisely, for a log Calabi-Yau pair $(X,D)$ with $D$ a simple normal crossings divisor, the main result of \cite{int_mirror} constructs an algebra $R_{(X,D)}$ over a monoid ring $S_X$ of effective curves classes $H_2^+(X)$. This $S_X$-algebra has a basis, called the theta basis, indexed by an appropriate subset of the integral points of the dual cone complex of $D$. In particular, there is a theta function $\vartheta_0$ which gives the identity element in $R_{(X,D)}$. For $W \in R_{(X,D)}$, denoting by $W[\vartheta_0]$ the coefficient of $\vartheta_0$ of $W$, we define the classical periods of $W$ by:
\[\pi_W = \sum_{d\ge 0} W^d[\vartheta_0] \in \kk[[H_2^+(X)]]\]

When $X$ is a smooth Fano variety, we define a collection of genus $0$ $1$-pointed Gromov-Witten invariants of $X$ indexed by $\textbf{A} \in \NE(X)$:
\[n_{\textbf{A}} = \int_{[\mathscr{M}_{0,1}(X,\textbf{A})]^{\vir}} \psi_{x}^{\textbf{A}\cdot c_1(X)-2}\ev_x^*([pt])\]
 We collect these invariants into a generating series $\widehat{G}_X$ called the regularized quantum periods of $X$:

\begin{equation}\label{regqp}
\widehat{G}_X = \sum_{d\ge 0} \sum_{ -K_X\cdot A = d} d!n_\textbf{A} t^\textbf{A} \in \kk[[H_2^+(X)]].
\end{equation}
The inner sum above is finite, and for $d \ge 0$, we let $n_d = \sum_{-K_X\cdot A = d} n_\textbf{A}$. Our main theorem is the following:

\begin{theorem}\label{mthm1}
Let $(X,D)$ be a pair with $X$ smooth Fano and $D \in |-K_X|$ a reduced representative. Suppose there exists a log Calabi-Yau pair $(X'',D'')$ in the sense of \cite{int_mirror} with $X''\setminus D'' \cong X\setminus D$. Then for any dominating snc compactification $(X',D')$ of $(X,D)$, there exists a function $W_{D} \in R_{(X',D')}$ which is mirror to the compactification $(X,D)$ of $U$, i.e. after an appropriate base change of the mirror family $\Spec\text{ }R_{(X',D')} \rightarrow \Spec\text{ }H_2^+(X')$:
\begin{equation}\label{weakmir}
\pi_W = \widehat{G}_X.
\end{equation}
 
\end{theorem}

\begin{remark}
\begin{enumerate}
\item Note that the condition of the existence of a pair $(X'',D'')$ imposes a restriction on the singularities of $D$. For instance, the theorem does not apply if we let $X = \mathbb{P}^2$ and $D$ the union of three lines intersecting at a point. This pair is resolved by blowing up the intersection point, but the resulting snc pair is not log Calabi-Yau.
\item Mandel in \cite{fanoperiod} has previously shown under various hypotheses that the regularized quantum periods smooth Fano variety are a naive curve count, see Theorem $1.2$ of op. cit. We recover and extend this result in our setting, as the $\vartheta_0$ structure constants for an affine log Calabi-Yau encode these naive counts by the main result of \cite{mirrcomp}.
\item In \cite{Tonk19}, Tonkonog proves an analogous result in symplectic geometry, in which the function $W_D$ of Theorem \ref{mthm1} is replaced with the Landau-Ginzburg potential of a monotone Lagrangian torus $L \subset X$. 

\end{enumerate}
\end{remark}

The independence of the general fiber of the mirror families $(X',D')$ above are known to be independent when $X\setminus D$ contains a dense torus by a combination of the main results of \cite{archmirror} and \cite{mirrcomp}. In general, since $X\setminus D$ is affine, this independence is expected but not yet known. In any case, Therorem \ref{mthm1} states that any choice of compactification yields a mirror Landau-Ginzburg pair in the weak sense described by Equation \ref{weakmir}. 

Using the mirror symmetry result of Theorem \ref{mthm1}, we derive the following corollaries relevant to classical mirror symmetry of Fano varieties:

\begin{corollary}
Assuming there exists $D \in |-K_X|$ reduced with $X\setminus D$ log Calabi-Yau, then the regularized quantum periods are non-negative integers. 
\end{corollary}

In applications to representation theory, many Fano varieties come in the form of a compactification of a cluster variety by an anticanonical divisor, for instance Grassmannians \cite{scottclust} and flag varieties \cite{braidclust}. In these settings, we derive the following corollary:

\begin{corollary}\label{qdiff}
Suppose $U$ is a skew-symmetric affine cluster variety over $\mathbb{C}$ with optimized seeds for each of the divisiorial valuations $\nu_{D_i}$ associated with $D_i \subset D$ in the sense of \cite[Definition $9.1$]{bases_cluster}. Given a seed torus $\mathcal{T}_s \subset U$ with cluster monomials $x_1,\ldots,x_n$, then $X$ has a mirror Laurent polynomial defined on the dual torus $\mathcal{T}^{\vee}_s$, i.e. a Laurent polynomial $W$ defining a function on the dual torus $\mathcal{T}_s^{\vee}$ such that the following oscillatory integral is a solution to the quantum differential equation for $X$:

\[g(t) = \frac{1}{(2\pi i)^n}\int_{\Gamma} e^{Wt} \bigwedge_j\frac{dx_j^*}{x_j^*}.\]
The integral above is computed using a dual seed $\mathcal{T}_s^{\vee} = (\mathbb{C}^n)^* \rightarrow \check{X}^\circ$ of the Fock-Goncharov dual cluster variety, with the homology of the cycle $\Gamma \in H_n((\mathbb{C}^n)^*,\ZZ)$ represented by an oriented compact torus $(S^1)^n$ which is the real form of $\mathcal{T}_s^{\vee}$.

\end{corollary}

We also associate a full-dimensional positive subsets in the essential skeleton of $\Sigma(X')$ in the sense of  \cite{bases_cluster} to every pairing of a seed of a cluster variety and a Fano compactification. These positive subset in turn induce toric degenerations of $X$. We explicitly compute the superpotential in the case $X = \Gr(n-k,n)$, recovering the Pl\"ucker coordinate mirror discovered by Marsh and Rietsch in \cite{LGgr}, and its relationship with Newton-Okounkov bodies and toric degenerations of $X$ discovered by Rietsch and Williams in \cite{NOmirr}.

%
%
%
%
%
In the final section, we will consider the situation where $D \in |-K_X|$ is a smooth divisor. In this setting, using Theorem \ref{mthm1} together with a study of how these invariants decompose after degeneration, we find the following collections of invariants are equivalent:

\begin{theorem}[Frobenius structure theorem for smooth pairs]\label{mcr2}
Let $(X,D)$ be a pair of a Fano variety together with a smooth anticanonical divisor. Then the following collection of (log) Gromov-Witten invariants are equivalent data:

\begin{enumerate}
\item The regularized quantum periods of $X$.
\item The structure constants $N_{p,q}^{r,\textbf{A}}$ of the log mirror algebra associated to $(X,D)$. 
\item The two pointed log Gromov-Witten invariants $N_{p,q}^{\textbf{A}} := \langle [pt]_p, [1]_q\rangle_{\textbf{A},2,0}$.
\end{enumerate}


\end{theorem}

Theta function structure constants for $R_{(X,D)}$ have previously been studied by Wang in \cite{Ysmooth} in the setting $X = \mathbb{P}^2$, for Del Pezzo surfaces by Gr\"afnitz, Ruddat and Zaslow in \cite{GRZ1} by scattering techniques and by You in \cite{YouPLG} without dimension restriction using the log-orbifold correspondence of \cite{logroot2} and an investigation of relative mirror maps. Theorem \ref{mcr2} essentially follows from the latter work, although it is not stated in this form. We derive this result by our main theorem combined with a degeneration argument, as opposed to an application of a mirror theorem.

%
%

%
%
%
%

\subsection{Future Work}
Kalashnikov in \cite{LPquiv} provides many candidate Laurent mirrors to Fano varieties arising as quiver flag zero loci. Central to constructing these mirrors is a toric degeneration of the ambient quiver flag variety. It is natural to ask whether we could generalize the construction of mirrors in the case of Grassmanians to the setting of Fano quiver flag varieties, along with the corresponding collection of Newton-Okounkov bodies and corresponding toric degenerations studied by Rietsch and Williams in \cite{NOmirr}. 

Additionally, Theorem \ref{mthm1} expresses the quantum periods of a smooth Fano variety $X$ in terms of the log Gromov-Witten invariants of another log smooth birational model $(\tilde{X},D) \rightarrow X$, and the proof the Theorem \ref{mthm1} establishes that these do not depend on our choice of desingularization. These results suggest a generalized notion of quantum periods which makes sense outside the context of a smooth Fano variety, for instance Schubert varieties in flag varieties. Using this generalized notion of quantum periods, it would be interesting to understand whether the ordinary quantum periods of a Fano variety are determined by the generalized quantum periods of a singular toric variety to which the Fano degenerates. More generally, since the quantum periods are the degree $0$ component of the Givental J-function for $X$ when $X$ is smooth, another natural question is to understand whether other parts of the J-function for Fano varieties admit a description in terms of log Gromov-Witten invariants. 

\section{Acknowledgments}
I would like to thank  George Balla, Alessio Corti, Michel Van Garrel, Mark Gross, and Konstanze Rietsch for useful discussions related to this work. This work was supported by the Additional Funding Programme for Mathematical Sciences, delivered by EPSRC (EP/V521917/1) and the Heilbronn Institute for Mathematical Research



\section{Logarithmic and tropical preliminaries}

In the following, we let $\textbf{Cones}$ be the category of rational polyhedral cones, and $\textbf{RPCC}$ be the category of rational polyhedral cone complexes. Among its objects are rational polyhedral cones $\sigma$, with associated integral points $\sigma_\NN \subset \sigma$ and lattice $\sigma_\NN \subset \sigma_{\NN}^{\gp}$. 

Given a simple normal crossings pair $(X,D)$, we let $\Sigma(X)$ denote the dual cone complex of the divisor $D$. Letting $D = D_1+ \cdots +D_k$ be the decomposition of $D$ into irreducible components, the cones of $\Sigma(X)$ are indexed by strata $D_I = X \cap_{i \in I} D_i$ associated with subsets $I \subset [1,k]$, giving a cone $\sigma_I \cong \mathbb{R}_{\ge 0}^{|I|}$. These cones are glued by identifying $\sigma_I$ with a face of $\sigma_J$ if $I \subset J$. We will frequently denote $\sigma_{\emptyset}$ by $\{0\}$.

To relate the PL geometry of $\Sigma(X)$ to the algebraic geometry of $X$, we recall that the data of the pair $(X,D)$ for $D = D_1+\cdots+D_k$ is encoded in a morphism of algebraic stacks $X \rightarrow [\mathbb{A}/\mathbb{G}_m]^k = \mathcal{A}^k$, i.e. a choice of $k$ pairs of line bundles and sections. We have the isomorphisms $A^*(\mathcal{A}^k) = A^*_{\mathbb{G}_m^k}(\mathbb{A}^k) \cong \kk[x_1,\ldots,x_k]$, and the pullback of class $x_i$ to $X$ gives the line bundle $\mathcal{O}_X(D)$. Moreover, there is a closed substack $im(X) = \mathcal{X} \subset \mathcal{A}^k$. When the intersections $D_I$ are connected, we call $\mathcal{X}$ the \emph{Artin fan} for the pair $(X,D)$. For an introduction to Artin fans more broadly, we refer the reader to \cite[Section $6$]{stacktrop}. We recall the following useful lemma which relates these algebraic stacks with appropriate combinatorial objects:

\begin{lemma}[\cite{stacktrop} Theorem $6.11$]\label{equiv}
There is an equivalence of $2$-categories between the categories of Artin fans and cone stacks. On a cone $\sigma \in Ob(\textbf{RPCC})$, the equivalence is given by sending $\sigma$ to $\mathcal{A}_\sigma = [\Spec\text{ }\kk[\sigma^{\vee}_\NN]/\Spec\text{ }\kk[(\sigma^{\vee}_\NN)^{\gp}]$. 
\end{lemma}

In particular, under the equivalence above, we have $\mathcal{X}$ is associated with the dual cone complex $\Sigma(X)$ above, and $[\mathbb{A}^1/\mathbb{G}_m]$ is associated with $\mathbb{R}_{\ge 0}$. Moreover, given a piecewise linear function $\rho: \Sigma(X) \rightarrow \mathbb{R}_{\ge 0}$, we have an associated morphism $\mathcal{X} \rightarrow [\mathbb{A}^1/\mathbb{G}_m]$, and after precomposing with $X \rightarrow \mathcal{X}$, we find that $\rho$ induces a Cartier divisor supported on $D$, with associated line bundle denoted by $\mathcal{O}(\rho)$. Conversely, given any Cartier divisor supported on $D$, one can construct a piecewise linear function which pulls back to the given Cartier divisor. In particular, for every component $D_i$ of $D$, we have an associated piecewise linear function $ \Sigma(X) \rightarrow \mathbb{R}_{\ge 0}$, which we also refer to as $D_i$. By restricting to a cone $\sigma \in \Sigma(X)$ and taking the induced morphism on associated groups, we also have the induced linear map $D_i: \sigma^{\gp} \rightarrow \mathbb{R}$.

Along with the tropicalization of the target $\Sigma(X)$ introduced above, we also recall the tropical analogue of prestable curves and maps:

\begin{definition}
An abstract tropical curve over a cone $\sigma \in \textbf{Cones}$ is the data $(G,\textbf{g},l)$ with
\begin{enumerate}
\item $G = G_\sigma$ a graph with legs, with set of vertices, edges and legs denoted by $V(G_\sigma)$, $E(G_\sigma)$ and $L(G_\sigma)$ respectively.
\item $\textbf{g}: V(G) \rightarrow \NN$ a function called the genus function.
\item $l: E(G) \rightarrow Hom(\sigma_\NN,\NN)\setminus 0$ and assignment of edge lengths. 
\end{enumerate}
\end{definition}

The above data naturally determines a cone complex $\Gamma_\sigma$ and a morphism $\pi: \Gamma_\sigma \rightarrow \sigma$ such that for $p \in \sigma$, we have $\pi^{-1}(p)$ is a metrization of $G_\sigma$ with the length of an edge $e \in E(G_\tau)$ given by $l(e)(p) \in \RR_{\ge 0}$. With this local model, it is straightforward to extend the definition of a family of tropical curves over a cone $\sigma$ to a family of tropical curves over a cone complex $\Sigma$.

%
%
%
%
%
%
%
%

With the tropical domains determined, we now define a tropical map with target $\Sigma(X)$ to simply be a morphism in \textbf{RPCC} from a tropical curve $\Gamma_\sigma$ over a cone $\sigma$ to $\Sigma(X)$. We note that any tropical map determines a \emph{tropical type}, recalled below:

\begin{definition}
Given a tropical map $\Gamma_\sigma \rightarrow \Sigma(X)$ over a cone $\sigma$, we define a tropical type $\tau$ to be the tuple $(G,\textbf{g},\pmb\sigma,\textbf{u})$ where:
\begin{enumerate}
\item $(G,\textbf{g})$ is a graph with vertices decorated by natural numbers underlying $\Gamma_\sigma$.
\item $\pmb\sigma: V(G)\cup E(G) \cup L(G) \rightarrow \Sigma(X)$ assigns a feature of $G$ the smallest cone of the complex $\Sigma(X)$ which the cone associated to the feature maps into.
\item $\textbf{u}(e) \in \pmb{\sigma}(e)^{\gp}_{\NN}$ for $e \in E(G) \cup L(G)$ is the integral tangent vector given by the image of the tangent vector $(0,1) \in \sigma^{\gp}_{e,\NN} = \sigma^{\gp}_\NN \oplus \mathbb{Z}$ under the map $h: \Gamma_\sigma \rightarrow \Sigma(X)$.
\end{enumerate}
Any tuple $(G,\textbf{g},\pmb\sigma,\textbf{u})$ with the data appearing above is called a tropical type. A family of tropical maps $\Gamma_{\sigma} \rightarrow \Sigma(X)$ over a cone $\sigma$ is marked by the tropical type $\tau$ if there is a face $\omega \subset \sigma$ such that $\Gamma_{\sigma}|_{\omega} \rightarrow \Sigma(X)$ is a family of tropical maps of type $\tau$. A decorated tropical type $\pmb\tau$ is a tropical type $\tau$ with the additional data of a function $\textbf{A}: V(G) \rightarrow M$ with $M$ a monoid of curve classes. 
\end{definition}

As we will be working in genus $0$, we will suppress the genus marking in a tropical type from now on. Given a tropical type $(G,\pmb\sigma,\textbf{u})$, the existence of an associated basic monoid shown in \cite[Proposition $2.32$]{punc} shows that there exists a cone $\tau$ and a tropical map $\Gamma_\tau \rightarrow \Sigma(X)$ with the given tropical type such that any family of tropical maps $\Gamma_\sigma \rightarrow \Sigma(X)$ of the given type over a cone $\sigma$ is induced via pullback from a morphism of cones $\sigma \rightarrow \tau$. The cone $\tau$ is thus a moduli space of tropical maps, and  we will abuse notation throughout to refer to the type and the cone by the same name $\tau$.

\subsection{Log Gromov-Witten Theory}

Beyond ordinary Gromov-Witten theory of a smooth projective variety $X$, we will additionally be interested understanding the enumerative geometry of pairs $(X,D)$ for $D \subset X$ a simple normal crossings compactification of $U:= X\setminus D$. In the present work, we will probe the enumerative geometry of such a pair via log Gromov-Witten theory of simple normal crossings compactifications $(X,D)$. Developed by Abramovich, Chen, Gross and Siebert in \cite{mtodf1},\cite{mtodf2},\cite{logGW},\cite{punc}, log Gromov-Witten theory is a logarithmic enhancement of classical Gromov-Witten theory, which provides a compactification of the moduli problems classifying stable  maps to $X$ with prescribed orders of tangency with the divisor $D$.

 A particular feature of the theory is the presence of a tropicalization functor. For any prestable log curve $C/S$, there is a family of tropical curve $\Gamma_{\Sigma(S)}$ over the tropicalization of $S$, $\Sigma(S)$. Moreover, given a family of log stable map $f: C/S \rightarrow X$, tropicalization yields a tropical map $\Gamma_{\Sigma(S)} \rightarrow \Sigma(X)$. 

As in ordinary Gromov-Witten theory, to produce finite type moduli stacks over which we integrate to produce invariants, we impose discrete data on our moduli problems. These include the standard data of degree and genus of the maps. Additionally, since the maps we are considering have log structure, we may further impose discrete data coming from the tropicalization of a log stable maps. In particular, given a tropical map $\Gamma_{\tau} \rightarrow \Sigma(X)$, we have an associated tropical type, and a corresponding moduli stack $\mathscr{M}(X,\tau)$ of punctured log maps $C/S \rightarrow X$ marked by $\tau$. Furthermore, by decorating the vertices of $G_\tau$ by a monoid $H_2^+(X)$ of effective curve classes, \cite[Section $2.2.2$]{punc} introduces the moduli stack of punctured log stable maps marked by $\pmb\tau$ to $X$, denoted by:
 \[\mathscr{M}(X,\pmb\tau).\]
In addition, for the purpose of defining a virtual class for $\mathscr{M}(X,\pmb\tau)$, assuming in addition a leg $l \in L(G_{\tau})$, we also consider the moduli stacks of prestable log curves marked by a tropical type $\tau$ or a decorated tropical type $\pmb\tau$ to the Artin fan $\mathcal{X}$ together with a point $z \in \pmb\sigma(l)$, denoted respectively by:
 \[\mathfrak{M}^{\ev}(\mathcal{X},\tau),\mathfrak{M}^{\ev}(\mathcal{X},\pmb\tau).\]

Fixing the discrete data as above, the stacks $\mathscr{M}(X,\pmb\tau)$, $\mathfrak{M}^{\ev}(\mathcal{X},\tau)$ and $\mathfrak{M}^{\ev}(\mathcal{X},\pmb\tau)$ all carry log structures, $\mathscr{M}(X,\pmb\tau)$ is proper and Deligne-Mumford, and $\mathfrak{M}^{\ev}(\mathcal{X},\tau)$ and $\mathfrak{M}(\mathcal{X},\pmb\tau)$ are idealized log smooth and equidimensional. Finally, the forgetful morphism $\mathscr{M}(X,\pmb\tau) \rightarrow \mathfrak{M}^{\ev}(\mathcal{X},\tau)$ carries a relative perfect obstruction theories, and virtual pullback in the sense of \cite{vpull} of the fundamental class $[\mathfrak{M}(\mathcal{X},\tau)]$ yield a virtual fundamental class $[\mathscr{M}(X,\pmb\tau)]^{\vir} \in A_*(\mathscr{M}(X,\pmb\tau))$.

Log Gromov-Witten invariants in turn are defined as in ordinary Gromov-Witten theory; there exist evaluation maps on underlying stacks $\mathscr{M}(X,\pmb\tau) \rightarrow X_{\pmb\sigma(l)}$ associated with legs $l \in L(G_\tau)$, and by pulling back classes along these evaluation maps, we produce log Gromov-Witten invariants by integrating against the virtual fundamental class. Analogous to the classical story, we may additionally consider descendent log Gromov-Witten invariants, in which the integrand may include insertions coming from the tautological $\psi$ classes, i.e. the first Chern classes of the tautological cotangent line associated with a marked point. 

An important constraint on the tropicalization of log stable maps is the logarithmic balancing condition, recalled below:

\begin{proposition}
Given a log stable map $C \rightarrow X$ with curve class $\textbf{A} \in \NE(X)$ of tropical type $\tau$, and $L$ a line bundle on $X$ corresponding to a PL function on $\Sigma(X)$:

\[\textbf{A}\cdot L = \sum_{l \in L(G_{\tau})} L(\textbf{u}(l))\]

\end{proposition}

Finally, we recall that if $\tau$ contains a leg $l$, then for $z \in X_{\pmb\sigma(l)}$, we can form the point constrained moduli stack $\mathscr{M}(X,\pmb\tau,z)$. When $z \in X\setminus D$, this stack is simply $\mathscr{M}(X,\pmb\tau)\times_{X_{\pmb\sigma(l)}} z$. For the virtual class, $\mathscr{M}(X,\pmb\tau,z) \rightarrow \mathfrak{M}^{\ev}(\mathcal{X},\tau,z) := \mathfrak{M}^{\ev}(\mathcal{X},\tau)\times_{X_{\pmb\sigma(l)}} z$ is also equipped with a perfect obstruction theory via pullback, and virtual pullback of the fundamental class of the codomain as done previously yields a virtual class $[\mathscr{M}(X,\pmb\tau,z)]^{\vir} \in A_*(\mathscr{M}(X,\pmb\tau,z))$. Denoting by $vdim\text{ }\mathscr{M}$ the virtual dimension of a moduli stack $\mathscr{M}$, we have:
\[vdim\text{ }\mathscr{M}(X,\pmb\tau,z) = vdim\text{ }\mathscr{M}(X,\pmb\tau) - dim\text{ }X_{\pmb\sigma(l)}.\]

\section{Intrinsic Mirror Symmetry for log Calabi-Yau pairs}
A large class of targets in the algebraic setting for which mirror symmetry is studied are log Calabi-Yau varieties, which following \cite{int_mirror} are defined as follows:

\begin{definition}
A log Calabi-Yau variety $U$ is a quasiprojective variety which admits a compactification $U \subset Y$ in a projective variety with $D= D_1 +\ldots + D_k =Y\setminus U$ a normal crossings divisor such that $K_Y + D = \sum_i a_iD_i$ with $a_i \ge 0$. We call such pairs $(Y,D)$ log Calabi-Yau pairs.
\end{definition}

In this paper, we will be working with Fano varieties $X$ with a reduced anticanonical divisor $D \in |-K_X|$, with $U$ log Calabi-Yau in the above sense. Moreover, we show in the following that any log resolution $(X',D') \rightarrow (X,D)$ of such a pair is log Calabi-Yau:

\begin{lemma}\label{logres}
If $U$ admits a log Calabi-Yau compactification $(Y,D'')$, then any snc compactification $(X',D')$ of $U$ is a log Calabi-Yau pair.
\end{lemma}

\begin{proof}
By weak factorization of birational maps i.e. \cite[Theorem $0.3.1$]{weakfact}, there is a sequence of blowups and blowdowns along smooth centers relating $(X',D')$ and $(Y,D'')$:

\[\begin{tikzcd}
(X',D') = (X_0,D_0) & (X_1,D_1)\arrow{l} \arrow{r}& (X_2,D_2) \cdots  &(X_{m-1},D_{m-1})\arrow{l}\arrow{r}& (Y,D'')
\end{tikzcd}.\]

Moreover, the centers of the blowup with base $(X,D_i)$ has normal crossings with $D_i$, and $D_i$ has at worst simple normal crossings singularities at each stage. By the blowup formula for the canonical divisor, if $(X_i,D_i)$ is log Calabi-Yau and $(X_{i+1},D_{i+1}) \rightarrow (X_i,D_i)$ is a blowup in the sequence above, then $(X_{i+1},D_{i+1})$ is also log Calabi-Yau. On the other hand, if we have $(X_i,D_i)\rightarrow (X_{i+1},D_{i+1})$ instead, then a meromorphic volume form on $(X_i,D_i)$ with at worst simple poles along each boundary component will have at worst simple poles along the strict transform of each component of $D_{i+1}$. It follows that $(X_{i+1},D_{i+1})$ is also a log Calabi-Yau pair as well. Induction then shows that $(X_i,D_i)$ is log Calabi-Yau for all $i$, which in particular gives the desired result. 
\end{proof}
Hence, by passing to an embedded resolution of singularities of $D \subset X$, we produce a simple normal crossings compactification of $U$, $(X',D')$, which is a log Calabi-Yau pair.
 
 Associated to the normal crossings pair $(X',D')$ is the dual intersection cone complex $\Sigma(X')$. An important subcomplex to consider when studying $X'$ as compactification of $U$ is the essential skeleton $B$:
 
 \begin{definition}
 The essential skeleton $B(U) = B \subset \Sigma(X')$ is a sub cone complex whose cones are given by $\sigma_I \in \Sigma(X')$ for $I \subset \{1,\ldots k\}$ such that $X'\cap_{i\in I} D_i'$ is non-empty and $a_i = 0$ for all $i \in I$. We call the divisors associated with the $1$-dimensional cones of $B$ \emph{good divisors}.
 \end{definition}

While a priori the topological space $B$ defined above depends on a choice of snc compactification of $U$, it turns out to depend only on $U$. Indeed, we can identify $U$ with a ``skeleton" of the Berkovich analytic space $U^{\an}$, with $U$ thought of as a variety over the non-archimedean field $\kk$ equipped with the trivial valuation. The volume form $\omega \in \Gamma(U,K_U)$ induces an upper semicontinuous function $|\omega|: U^{\an} \rightarrow \mathbb{R}_{\ge 0}$ via Temkin's Kahler seminorm, see \cite[Section $8$]{Kmet}  for details. The locus on which this function is maximized can be identified with the essential skeleton of $B$ defined above for a given choice of normal crossings compactification, see \cite{k3affine} and \cite{NX16}. In addition, given a birational morphism $Y' \rightarrow Y$ which identify a shared open log Calabi-Yau $U$, and $D$ a prime divisor contained in $Y\setminus U$, the resulting map $Y^{'an} \rightarrow Y^{\an}$ induces an isomorphism $B' \rightarrow B$ which maps the divisorial valuation associated with the strict transform of the divisor $D$ to the divisorial valuation $\nu_D$. We record this fact in the following proposition:

\begin{proposition}
The underlying topological space and set of integral points of the cone complex $B \subset \Sigma(X)$ for a simple normal crossings compactification $U \subset X$ depends only on $U$.
\end{proposition}
By the proposition above, we will implicitly identify the essential skeletons for different compactifications throughout. We observe that $B \subset \Sigma(Y)$ admits an integral structure, and we denote by $B(\NN)$ the integral points of $B$.

In the mirror constructions of Gross and Siebert in \cite{int_mirror} and Keel and Yu in \cite{archmirror}, the starting data is a choice of an snc compactification $U \subset X$. Such a compactification is a log smooth variety, and hence the enumerative geometry of the pair $(X,D)$ can be studied using log Gromov-Witten theory. In \cite{int_mirror}, Gross and Siebert define an algebra $R_{(X,D)}$ defined over a monoid ring $\kk[Q]$ of curve classes in $X$ with underlying free $\kk[Q]$ module structure $R_{(X,D)} = \oplus_{p \in B(\NN)} \kk[Q]\vartheta_p$. We refer to the free basis generating $R_{(X,D)}$ as a $\kk[Q]$ module the \emph{theta basis}. This algebra is defined using structure constants for the product in the theta basis. Given $p_1,\ldots,p_m ,r \in B(\NN)$ and $\textbf{A} \in Q$, we denote by $N_{p_1,\ldots,p_m}^{r,\textbf{A}}$ the coefficient of $\vartheta_rz^\textbf{A}$ in the theta function expansion of the product $\vartheta_{p_1}\cdots\vartheta_{p_m}$. 

The structure constants $N_{p_1,\ldots,p_m}^{0,\textbf{A}}$ will play a central role; when $U$ is affine and $D$ contains a zero stratum, they turn out to be certain naive curve counts considered by Keel and Yu considered in \cite{archmirror}. To recall these naive curve counts, we consider the space of maps $H(X,p_1,\ldots,p_m,\textbf{A})$ consisting of maps $f: (\mathbb{P}^1,x_1,\ldots,x_m,x_{\out}) \rightarrow X$ which intersect the boundary at finitely many points with contact order at a marked points $x_i$ determined by $p_i$, and $f_*[\mathbb{P}^1] = \textbf{A}$. There is a forgetful morphism $\Phi_{(p_i),\textbf{A}}: H(X,(p_i),\textbf{A})\rightarrow \overline{\mathscr{M}}_{0,m+1}\times X$ defined by $\Phi_{(p_i),\textbf{A}}((C,f)) = (C,f(x_{\out}))$. By \cite[Theorem $3.12$]{archmirror}, this map is finite \'etale over a open subset of the codomain, and we define:

\[\eta(p_1,\ldots,p_m,\textbf{A}) = deg(\Phi_{(p_i),\textbf{A}}).\]

When $D$ is the support of a nef divisor, contains a zero stratum and $U$ is affine, the following theorem from \cite{mirrcomp} relates the structure constants in the mirror algebra to the naive curve counts above:

\begin{theorem}[Theorem $1.1$ \cite{mirrcomp}]\label{enum}
For $(X,D)$ as above, with $p_1,\ldots,p_m \in B(\NN)$ and $\textbf{A} \in \NE(X)$, we have:

\[N_{p_1,\ldots p_m}^{0,\textbf{A}}= \eta(p_1,\ldots,p_m,\textbf{A})\]

\end{theorem}

In addition, with only the assumption that $(X,D)$ is log Calabi-Yau, it is shown in \cite{logWDVV}, that we have the following descendent log Gromov-Witten description of the $\vartheta_0$ structure constants of $R_{(X,D)}$:

\begin{proposition}\label{logTRR}
For $m \ge 2$ and $\vartheta_{p_1},\ldots,\vartheta_{p_m}$ theta functions in $R_{(X,D)}$, then letting $\textbf{P}$ be the contact order vector $(p_1,\ldots,p_m,0)$ we have:
\[N_{p_1,\ldots p_m}^{0,\textbf{A}}= \int_{[\mathscr{M}(X,\textbf{P},A)]^{\vir}} \psi_{x_{\out}}^{m-2}ev_{x_{\out}}^*([pt])\]

\end{proposition}

A combination of the previous two results in the more restrictive setting of Theorem \ref{enum} immediately yields the enumerativity of certain structure constants in $R_{(X,D)}$. A slight generalization of this fact will be useful in a degeneration analysis which will appear in the next section, see Lemma \ref{primnaive}.

\section{Mirrors to Fano pairs (X,D)}

Letting $(X,D)$ be a pair of a Fano variety and $D \in |-K_X|$ satisfying the conditions of Theorem \ref{mthm1}, we fix a choice of log resolution $(X',D') \rightarrow (X,D)$ which by a combination of the conditions of Theorem \ref{mthm1} together with Lemma \ref{logres} we know is log Calabi-Yau. The pair $(X',D')$ will be used in the proof of Theorem \ref{mthm1}. By pulling back the Cartier divisor $D \subset X$ to $X'$, we produce an effective Cartier divisor with support $D'$. In particular, as described in the remarks following Lemma \ref{equiv}, the associated line bundle is induced by a piecewise linear function on $\Sigma(X')$, which we also denote by $D$. 

Recall now that for $z \in X\setminus D$, we may form the moduli stack $\mathscr{M}(X',\beta,z) = \mathscr{M}(X',\beta)\times_{X'} z$ of log stable maps of marked by $\beta$ which map $x_{\out}$ to $z$, equipped with the virtual class $[\mathscr{M}(X',\beta,z)]^{\vir} = z^!([\mathscr{M}(X',\beta)]^{\vir})$. Using this stack we note the following slight extension of Theorem \ref{enum} which removes the assumption of the maximal boundary in certain special cases:

\begin{lemma}\label{primnaive}
Letting $(X,D),(X',D')$ be pairs as above, and for a multiset of points $p_1,\ldots,p_m \in \Sigma(X')(\NN)$ given by primitive points along $1$-dimensional cones of $\Sigma(X')$ associated with strict transforms of components of $D$, consider the type $\beta= ((0,p_1,\ldots,p_m),\textbf{A})$, with $x_{\out}$ denoting the contact order $0$ marked point. Then for $z \in X\setminus D = X' \setminus D'$ general, $[\mathscr{M}(X',\beta,z)]^{\vir}\cap (\psi_{x_{\out}}^{m-2})$ is represented by the zero cycle given by the finitely many curves contributing to $\eta(p_1,\ldots,p_m,\textbf{A})$. In particular, we have:
\begin{equation}\label{ncount}
\vartheta_{p_1}\cdots\vartheta_{p_m}[\vartheta_0z^{\textbf{A}}] = \eta(p_1,\ldots,p_m,\textbf{A}).
\end{equation}
\end{lemma}

\begin{proof}
Since $(X',D')$ is log Calabi-Yau, the intrinsic mirror algebra $R_{(X',D')}$ exists. By Proposition \ref{logTRR}, together with a straightforward comparison of the obstruction theories for $\mathscr{M}(X',\beta,z)$ and $\mathscr{M}(X',\beta)$, the left hand side of Equation \ref{ncount} is the log Gromov-Witten invariant:
\[N_{p_1,\ldots,p_m}^{0,\textbf{A}}  = \int_{[\mathscr{M}(X',\beta,z)]^{\vir}} \psi_{x_{\out}}^{m-2}\]
We wish to show this invariant is equal to $\eta(p_1,\ldots,p_m,\textbf{A})$. To do this, we first show that no component of any log stable map $f:C \rightarrow X'$ of type $\beta$ potentially contributing to the invariant above maps into the boundary. Tropically, this means that for $\tau$ the tropical type of a map $f$ as above, every compact edge of $G_\tau$ is contracted by $\Sigma(f)$ to $0 \in \Sigma(X)$. The desired equality would then follow as in the final section of the proof of \cite[Theorem $1.1$]{mirrcomp}.  

%
%

Let $v_{\out}$ be the vertex contained in the contracted leg associated with $x_{\out}$. As in Step $2$ of proof of  \cite[Theorem $1.1$]{mirrcomp}, the generic point constraint ensures that the component containing $x_{\out}$ is not contracted by stabilization of the domain curve, and we may replace the cohomology class $\psi_{x_{\out}}$ appearing the definition of $N_{p_1,\ldots,p_m}^{0,\textbf{A}}$ with $\Forget^*\overline{\psi}_{x_{\out}}$, where $\Forget: \mathscr{M}(X',\beta,z) \rightarrow \overline{\mathscr{M}}_{0,m+1}$ is the stabilization map. Since $\overline{\psi}_{x_{\out}}^{m-2}  = [pt] \in A^{m-2}(\overline{\mathscr{M}}_{0,m+1})$, a map contributes to $N_{p_1,\ldots,p_m}^{0,\textbf{A}}$ only if the domain stabilizes to $\mathbb{P}^1$ with a general configuration of points, with the component $C_{v_{\out}}$ associated with $v_{\out}$ not contracted by the stabilization morphism. Hence $C_{v_{\out}}$ contains $m$ special points other than $x_{\out}$. Moreover, the proof of \cite[Lemma $3.11$]{archmirror} also applies in our context to show that $C_{v_{\out}}$ does not share a contact order $0$ node with a component meeting the curve at only one point. It follows that the $m$ remaining special points must be either legs or non-contracted edges.




Now let $e$ be a non-contracted edge containing $v_{\out}$, with associated node $q_e \in C_{v_{\out}}$, and $G'$ be the tropical curve produced by cutting $G_\tau$ at the edge $e$ and taking the resulting connected component not containing $v_{\out}$, which exists since $G_\tau$ is genus $0$. By the established constraint on the stabilization of the domain curve, $G'$ can have at most two legs, the leg associated with the cut edge $e$, and another leg $l_i$ inherited from the graph $G_\tau$. Since the edge $e$ contained a vertex mapping to $0 \in \Sigma(X')$, the contact order of the leg of $G'$ associated with $e$ must satisfy $D_i(\textbf{u}(e)) \le 0$ for all $i$, and $D_j(\textbf{u}(e)) < 0$ for some $j$. In particular, $D(\textbf{u}(e)) < 0$. 

If $G'$ contained no leg inherited from $G_\tau$, then letting $C_{G'} \subset C$ be the punctured log curve given by restricting the log structure to the subcurve associated with $G'$, and $\textbf{A}'$ the curve class of $f': C_{G'} \rightarrow X'$, we have by the log balancing condition
\[D \cdot \textbf{A}' = \sum_i a_iD_i(\textbf{u}(e)) < 0.\]
But $D$ is an ample divisor on $X$, hence its pullback is nef on $X'$, so this would give a contradiction.

 Hence, $G'$ must have a leg inherited from $G_{\tau}$, with contact order $p_i$. Since $D(p_i) = 1$ for all $i$ by assumption, by the log balancing condition and the fact that $D(\textbf{u}(e))$ is a negative integer, the curve class $\textbf{A}'$ must satisfy $D\cdot \textbf{A}' \le 0$. Since $D$ is ample, the resulting map $C' \rightarrow X$ must be contracted, hence in particular $D \cdot \textbf{A}' = 0$. By the balancing condition for punctured log curve $C'$ associated with the tropical curve $G'$, we must have 
 \[D(\textbf{u}(e)) = - D(p_i) = -1.\]
Since an analogous equality holds for all non-contracted edges containing $v_{\out}$, we have $f(C_{v_{\out}})$ intersects $D'$ transversely. Now let $D'_i \subset D'$ be the irreducible component associated with the contact order $p_i$, which by assumption is the strict transform of a component $D_i \subset D$. Note that the complement in $D_i'$ of the isomorphism locus of $D' \rightarrow D$ is codimension at least $2$ in $X'$. By \cite[Lemma $3.9$]{archmirror}, after ensuring we have chosen $f(x_{\out}) = z \in X'$ sufficiently general, we can ensure that $f(q_e) \in D_i'$ is contained in the isomorphism locus of $D' \rightarrow D$ for any given choice of curve class $\textbf{A}'$. By varying over the finitely many curve classes which have degree less than the degree of $\textbf{A}$ with respect to some polarization of $X$, this can be ensured across the entire moduli stack. Since $C' \rightarrow X$ was shown to be a contraction, it follows that the map $C' \rightarrow X'$ must also contract $C'$. But $C'$ has only two marked points, so this would contradict stability. Thus, $v_{\out}$ can only be contained in legs of $G_{\tau}$. In particular, since the graph $G_{\tau}$ is connected, no vertex of $G_{\tau}$ can map into the boundary, as desired.

\end{proof}

Having described the types of curves which contribute to the trace form assuming primitive inputs, we now define $W_D \in R_{(X',D')}$ referred to in the statement of Theorem \ref{mthm1}, and proceed to its proof.

\begin{definition}
Letting $D = \sum_i D_i$, note that the divisorial valuation for each component $D_i$ gives an integral point in $B(U)$. Since $B(U) \subset \Sigma(X')$ for any snc compactification $(X',D')$ of $U$, we have a theta function $\vartheta_{D_i}$ associated with each component $D_i$, and we define:

\begin{equation}
W = W_D = \sum_i \vartheta_{D_i} \in R_{(X',D')}.
\end{equation}
\end{definition}

We now proceed to the proof of Theorem \ref{mthm1}. The main idea is that the quantum periods of a smooth Fano variety $X$ turn out to depend only on the naive curve counts which govern the trace form of the mirror algebra, which in turn is a naive count of curves contained in an open subscheme $V\subset X$ with $V$ a partial compactification of an affine log Calabi-Yau $U \subset Y$ and $V\setminus U$ a smooth divisor.

%
%
%
%
%
%
%

\begin{proof}[Proof of Theorem \ref{mthm1}]
We first reduce to the case that $(X',D')$ is an embedded resolution of singularities of $D \subset X$. We note that we have an analogous element $W_D \in R_{(X',D')}$ for any snc compactification of $U$. We now claim that the sequence $\pi_{W_D}$ appearing in the statement of Theorem \ref{mthm1} is independent of the choice of dominating compactification of $(X,D)$. To see this, consider a compactification $(\tilde{X},\tilde{D})$ of $U$ which dominates $(X,D)$, and let $p_{\tilde{X}}: \kk[\NE(\tilde{X})] \rightarrow \kk[\NE(X)]$ be the morphism of monoid rings induced by proper pushforward map. By Lemma \ref{primnaive}, we have that the log Gromov-Witten invariants contributing to $p_{\tilde{X}}(W_D^d[\vartheta_0])\in \kk[\NE(X)]$ are naive curve counts $\eta(p_1,\ldots,p_k,\textbf{A})$. Since the previous curve counts are independent of the choice of compactification by definition, the desired independence of $\pi_{W_D}$ on the choice of dominating compactification follows.

 Hence, it suffices to prove the equality of Theorem \ref{mthm1} in the case that $(X',D')$ is an embedded resolution of singularities of $(X,D)$. Hence, we write $X' = X_m$ as an iterated blowup of $X = X_0$, with the $i^{th}$ center $Z_i \subset X_{i-1}$ transverse to the exceptional locus of the birational morphism $X_{i-1} \rightarrow X$. 

We now consider a degeneration of $X$ constructed by a sequence of blowups centered on the special fiber of the trivial family $X \times\mathbb{A}^1$ which contains a main component isomorphic to $X'$. More precisely, after $i-1$ blowups, we have the corresponding birational morphism $(X\times \mathbb{A}^1)_{i-1} \rightarrow X \times \mathbb{A}^1$, and we consider the strict transform of the divisor $X \times 0$, giving a closed immersion $X_{i-1} \rightarrow (X\times \mathbb{A}^1)_{i-1}$. We identify $Z_i \subset X_{i-1}$ with a closed subvariety of $(X\times \mathbb{A}^1)_{i-1}$, and we blow up $Z_i$ to produce $(X\times \mathbb{A}^1)_{i}$. Since the centers $Z_i$ are transverse to the exceptional divisors of previous blow ups, it follows that the divisor over $0 \in \mathbb{A}^1$ has simple normal crossings, hence the total space is log smooth. The strict transform of the central fiber of the trivial degeneration is isomorphic to $X'$, which we call the main component of the special fiber. Finally, we blowup the strict transform of $D \times 0 \subset X \times \mathbb{A}^1$, which we note is a smooth divisor of the main component of the special fiber. Letting $\mathscr{X}$ denote the total space resulting from these blowups, we have a map $\mathscr{X} \rightarrow \mathbb{A}^1$ whose general fiber is $X$. Since the strict transform of $D\times 0$ intersects the exceptional locus transversely, it follows that the divisor over $0$ still has normal crossings after this additional blowup. Therefore, $\mathscr{X}$ is log smooth. Since the map $\mathscr{X} \rightarrow \mathbb{A}^1$ is \'etale locally induced by a choice of section of $\mathcal{M}_{\mathscr{X}}$, the family is log smooth by Kato's criterion for log smoothness. 


The main result of \cite{decomp} now allow us to compute our desired invariant on the general fiber via a log Gromov-Witten calculation on the special fiber. More precisely, letting $\textbf{A}\cdot c_1(X) = d$, we have:

\begin{equation}\label{degform}
n_\textbf{A} = \sum_{\pmb\tau}\frac{m_{\pmb\tau}}{|Aut(\pmb\tau)|}N_{\pmb\tau} := \sum_{\pmb\tau}\frac{m_{\pmb\tau}}{|Aut(\pmb\tau)|} \int_{[\mathscr{M}(\mathscr{X}_0,\pmb\tau)]^{\vir}} \psi_{x_{\out}}^{d-2}\ev^*([pt]).
\end{equation}

The above sum varies over all decorated rigid tropical types $\pmb\tau$ of log maps to the special fiber such that $\sum_{v\in V(G_\tau)} \pi_*(\textbf{A}(v)) = \textbf{A}$.  In what follows, we let $\pmb\tau$ be a rigid tropical type with non-trivial contribution to $n_\textbf{A}$. First, we note by cutting along all edges of $\tau$, we produce a tropical type $\tau_v$ for every vertex $v \in V(G_\tau)$. Let $v_0$ denote the unique vertex contained in the contracted leg associated with the marked point $x_{\out}$. We have a corresponding cutting morphism $\mathscr{M}(\mathscr{X}_0,\pmb\tau) \rightarrow \mathscr{M}(\mathscr{X}_0,\pmb\tau_{v_0})$, and the point and $\psi$ class both pullback from the corresponding class of the codomain. Moreover, $[pt] \in H^*(X)$ pulls back to the class $[pt] \in H^*(X')$ which in turn pulls back to $0$ for any closure of stratum of $X'$ not equal to $X'$. It follows then that the point constraint forces any contributing type $\pmb\tau$ to have a vertex $v_0$ mapping to the ray of $\Sigma(\mathscr{X}_0)$ associated with the main component of the special fiber.

 We now restrict our attention to $\pmb\tau_{v_0}$, which we will refer to as $\pmb\tau_0$. By \cite[Theorem $6.1$]{trglue}, we may identify the resulting space of punctured log maps with the space of stable log maps $\mathscr{M}(X',D',\textbf{P})$, where $\textbf{P} = (p_i)$ is a vector of non-negative contact orders with $D' \subset X'$. Observe that the underlying family of stable map associated with $\mathscr{M}(\mathscr{X}_0,\pmb\tau_0) \rightarrow \mathscr{M}(X',D',\textbf{P})$ is the underlying stable map of the universal family. In particular, the class $\psi_{x_{\out}}$ pulls back from the space of log maps. Similarly, the point constraints pulls back from the space of log maps. Moreover, as in the proof of \cite[Lemma $3.11$]{archmirror}, the generic point constraint ensures that the component containing the constrained point meets the boundary more than once. Hence, for $\pmb\tau$ to contribute, $v_0$ must be contained in at least two non-contracted edges. In particular, $\psi_{x_{\out}}$ pulls back from the class $\overline{\psi}_{x_{\out}} \in A^1(\overline{\mathscr{M}}_{0,\val(v_0)})$ along the stabilization map.

Since $dim\text{ }\overline{\mathscr{M}}_{0,\val(v_0)} = \val(v_0) - 3$, the integrand $\psi_{x_{\out}}^{d -2}\ev^*([pt])$ on the space of log stable maps to $(X',D')$ vanishes if $\val(v_0) < d+1$. To further restrict $\val(v_0)$, we pullback the ample divisor $D \subset X$ to $X'$ to produce an effective nef divisor with support $D' \subset X'$. In particular, we have $D'(\textbf{u}(l_i)) > 0$ for all non-contracted $l_i \in L(G_{\tau_{0}})$. If $\val(v_0) > d+1$, then the anticanonical degree of $\pmb\sigma(v_0)$ would be greater than $d$. Since the pullback of the anticanonical Cartier divisor is nef on $X'$, the total degree of $\pmb\tau$ would be greater than $d$, contradicting our starting numerical data. Thus, $\val(v_0) = d+1$. By similar reasoning, we also have $D'(\textbf{u}(l_i)) = 1$ for all legs $l_i \in L(G_{\tau_{0}})$. This latter condition can only be the case if $l_i$ corresponds to a primitive point in $\Sigma(X')$ associated with the strict transform of a component $D_i \subset D$. In particular, note that these integral points are contained in the essential skeleton $B \subset \Sigma(X')$.

\begin{figure}[h]
\centering
\begin{tikzpicture}
\draw[-] (0,0)--(-2,-2);
\draw[-] (0,0)--(2,0);
\draw[-] (0,0)--(0,2);
\draw[-] (-2,-2)--(2,0);
\draw[-] (-2,-2)--(0,2);
\draw[-] (2,0)--(0,2);
\fill[white!60!blue, path fading = south] (0,0)--(-2,-2)--(2,0)--cycle;
\fill[white!60!blue, path fading = north] (0,0)--(-2,-2)--(0,2)--cycle;
\fill[white!60!blue, path fading = east] (0,0)--(2,0)--(0,2)--cycle;
\draw[ball color = red] (0,0) circle (0.1cm);
\draw[ball color = green] (-2,-2) circle (0.1cm);
\draw[ball color = green] (2,0) circle (0.1cm);
\draw[ball color = green] (0,2) circle (0.1cm);
\draw[ball color = blue] (1,1) circle (0.1cm);
\draw[ball color = blue] (0,-1) circle (0.1cm);
\draw[ball color = blue] (-1,0) circle (0.1cm);
\draw[-, color = red] (0,.1)--(-2,-1.9);
\draw[-, color = red] (0,0)--(0,.1);
\draw[-, color = red] (-2,-2)--(-2,-1.9);
\draw[-, color = red] (0,0)--(0,-.1);
\draw[-,color = red] (0,-.1)--(-2,-2.1);
\draw[-,color = red] (-2,-2)--(-2,-2.1);
\draw[-, color = red] (0,0)--(.1,0);
\draw[-, color = red] (.1,0)--(.1,2);
\draw[-, color = red] (.1,2)--(0,2);
\draw[-, color = red] (0,0)--(-.1,0);
\draw[-, color = red] (-.1,0)--(-.1,2);
\draw[-, color = red] (-.1,2)--(0,2);
\draw[-, color = red] (0,0)--(0,.1);
\draw[-, color = red] (0,.1)--(2,.1);
\draw[-, color = red] (2,.1)--(2,0);
\draw[-, color = red] (0,0)--(0,-.1);
\draw[-, color = red] (0,-.1)--(2,-.1);
\draw[-, color = red] (2,-.1)--(2,0);
\end{tikzpicture}
\caption{The dual complex of the special fiber for the degeneration in the case $(X,D) = (\mathbb{P}^2,V(xyz))$, with a contributing tropical type depicted. The unique vertex mapping to the main component is colored red, while the components mapping to projective bundles over components of $D$ are colored green. Vertices colored blue correspond to irrelevant components with respect to a contributing type $\tau$.}
\end{figure}

 In order for $\pmb\tau$ to be a tropical type contributing to $N_\textbf{A}$, any edge of $G_\tau$ containing $v_0$ must have another vertex $v \in V(G_\tau)$, which maps to the convex hull of the vertices of the dual complex of $\mathscr{X}_0$ associated with the main component and a vertex associated with an additional component of the special fiber. This additional component of the special fiber is a projective bundle $\mathcal{P}_i$ over a component $D_i'\subset D'$. Moreover, by construction of the degeneration $\mathscr{X}$, the log structure on these components has generic stalks of $\overline{\mathcal{M}}_{\mathscr{X}}$ given by $\NN$, and the stalks jump in rank along the codimension $1$ locus given by the intersection with the main component of the special fiber, as well as the fibers of the projective bundles over intersections of components of $D'$. Since $D\subset X$ is ample, all vertices in $V(G_\tau)\setminus \{v_0\}$ must be decorated by a curve class which projects to the contracted curve class under $\mathscr{X}_0 \rightarrow X$, else the total degree would be greater than $d$. 
 
 We now claim that the curve class $\textbf{A}(v)$ must be a fiber class of the projective bundle $\mathcal{P}_i$ and $\val(v) = 1$ if $\pmb\tau$ is a contributing type. To see this, we note we can compute $N_{\pmb\tau}$ by replacing $\mathscr{M}(\mathscr{X}_0,\pmb\tau)$ with the moduli stack $\mathscr{M}(\mathscr{X}_0,\pmb\tau)_{(p,C)}$ for $(p,C) \in U\times \mathscr{M}_{0,k+1}$ general, defined by the following cartesian diagram:
 
 \[\begin{tikzcd}
\mathscr{M}(\mathscr{X}_0,\pmb\tau)_{(p,C)} \arrow{r} \arrow{d} & \mathfrak{M}^{\ev}(\mathcal{X}_0,\tau)_{(p,C)}\arrow{d}\arrow{r} & \Spec\text{ }\kk\arrow{d}\\
\mathscr{M}(\mathscr{X}_0,\pmb\tau) \arrow{r} & \mathfrak{M}^{\ev}(\mathcal{X}_0,\tau)\arrow[r,"\ev_{x_{\out}}\times Stab"] & X' \times \overline{\mathscr{M}}_{0,k+1}
\end{tikzcd}\]

Using this description, we equip $\mathscr{M}(\mathscr{X}_0,\pmb\tau)_{(p,C)}$ with the pulled back obstruction theory. With this perfect obstruction theory, $\mathscr{M}(\mathscr{X}_0,\pmb\tau)_{(p,C)}$ comes equipped with a virtual fundamental class by taking the Gysin pullback of the virtual class $[\mathscr{M}(\mathscr{X}_0,\pmb\tau)]^{\vir}$ along the regular closed immersion $\mathscr{M}(\mathscr{X}_0,\pmb\tau)_{(p,C)} \rightarrow \mathscr{M}(\mathscr{X}_0,\pmb\tau)$ induced by $\Spec\text{ }\kk \rightarrow X' \times \overline{\mathscr{M}}_{0,k+1}$. Moreover, the virtual dimension of $\mathscr{M}(\mathscr{X}_0,\pmb\tau)_{(p,C)}$ is zero, and $deg[\mathscr{M}(\mathscr{X}_0,\pmb\tau)_{(p,C)}]^{\vir} = N_{\pmb\tau}$. 

By Lemma \ref{primnaive}, any curve corresponding to a point of $\mathscr{M}(\mathscr{X}_0,\pmb\tau)_{(p,C)}$ must have component $x_{\out} \in C_{v_{\out}}$ mapping to the main component of $\mathscr{X}_0$ which intersects the boundary of the main component transversely and in the isomorphism locus of $X' \rightarrow X$. Therefore, by the degree constraint established previously, the projection of the image of the subcurve marked by $v \in V(G_{\tau})$ along $\mathcal{P}_i \rightarrow D_i$ must be contracted. Letting $C_v$ be the component of the subcurve which contains the node associated with the edge containing $v_{\out}$, the curve class of $C_v$ must in particular be a multiple of the fiber class. By the transverse contact condition at the node, the multiple must be $1$, so $C_v$ must have the curve class of a fiber. By the log balancing condition and the fact the genus of $G_{\tau}$ is zero, $v$ must only be contained in a single edge with primitive contact order. Since any other component marked by $v$ must also be a fiber curve class, which has positive intersection with the zero section of $\mathcal{P}_i$, we have that $C_v$ is the only component marked by $v$ and $\pmb\sigma(v)$ is a primitive fiber class.

 
In total, for $\pmb\tau$ to contribute non-trivially to Equation \ref{degform}, $\tau$ must be torically transverse, there must exist a vertex $v_0$ decorated with the curve class $\textbf{A}$ on the central component, and all other vertices are decorated with fiber curve classes of projecive bundles over componentes $D'_i \subset D'$. To compute the contribution $N_{\pmb\tau}$, we note that since every map corresponding to a geometric point $[C,f] \in \mathscr{M}(\mathscr{X}_0,\pmb\tau_0)_{(p,C)}$ intersects the boundary of the main component of $\mathscr{X}_0$ transversely, the following diagram is cartesian in all categories:

%


\[\begin{tikzcd}
\mathscr{M}(\mathscr{X}_0,\pmb\tau)_{(p,C)} \arrow{r}\arrow{d} & \mathscr{M}(\mathscr{X}_0,\pmb\tau_0)_{(p,C)} \times (\prod_{v \in V(G_\tau)\setminus \{v_0\}} \mathscr{M}(\mathscr{X}_0,\pmb\tau_v))\arrow{d}\\
\mathfrak{M}^{\ev}(\mathcal{X}_0,\tau)_{(p,C)} \arrow{r}\arrow{d} & \mathfrak{M}^{\ev}(\mathcal{X}_0,\tau_0)_{(p,C)} \times (\prod_{v \in V(G_\tau)\setminus \{v_0\}} \mathfrak{M}^{\ev}(\mathcal{X}_0,\tau_v))\arrow{d}\\
 \prod_{e \in E(G_\tau)}X_{\pmb\sigma(e)}' \arrow{r} & \prod_{e \in E(G_\tau)}X_{\pmb\sigma(e)}'\times X_{\pmb\sigma(e)}'
\end{tikzcd}\]


Indeed, all tuples of log curves coming from a point of the product moduli space which schematically glue are transverse with respect to the log structure on $\mathscr{X}_0$, and the desired cartesian diagram follows by \cite[Theorem $5.9$]{trglue}. Furthermore, by Lemma \ref{primnaive}, $\mathscr{M}(\mathscr{X}_0,\pmb\tau_0)_{(p,C)}$ contains only curves mapping to the main component whose underlying maps to $(X',D')$ contribute to the naive count $\eta(p_1,\ldots,p_d,\textbf{A})$. Moreover, using \cite[Theorem $5.19(1)$]{punc} and unwinding obstruction theories, we find that the obstruction theory on $\mathscr{M}(\mathscr{X}_0,\pmb\tau)_{(p,C)}$ pulls back from the obstruction theory for the product. Using Lemma \ref{primnaive} and the trivial obstruction theory for genus $0$ degree $1$ covers of fibers of a projective line bundle, the pushforward of the virtual class of the product along the right most vertical in the cartesian diagram above is $\langle\prod_i \vartheta_{D_i}^{p_i}\rangle[t^\textbf{A}][\prod_e ([pt]\times 1)]$. Pulling this class along the diagonal yields the pushforward of $[\mathscr{M}(\mathscr{X}_0,\pmb\tau)_{(p,C)}]^{\vir}$ along the left composite vertical morphism, giving a degree $([\mathscr{M}(\mathscr{X}_0,\pmb\tau)_{(p,C)}]^{\vir}) = \langle\prod_i \vartheta_{D_i}^{p_i}\rangle[t^\textbf{A}]$ zero cycle.

Letting $D_i\cdot \textbf{A} = q_i$, note that the corresponding tropical type $\tau$ has automorphism group of order $q_1!\cdots q_k!$ over the type with one vertex and no legs, and the multiplicity $m_{\pmb\tau}$ is $1$ since all edges in $\tau$ have primitive contact orders. It follows that:
 
 \[n_{\textbf{A}} =  \frac{1}{q_1!\cdots q_k!} \langle\prod_i\vartheta_{D_i}^{q_i}\rangle[t^\textbf{A}]\]
 
 Now consider the element $\frac{1}{d!}tr((\sum_i \vartheta_{D_i})^{d})$ in the intrinsic mirror algebra. Expanding this expression gives:
 \[\frac{1}{d!}\sum_{q_1+\ldots + q_k = d} \binom{d}{q_1,\ldots,q_k} tr(\vartheta_{D_1}^{q_1}\cdots \vartheta_{D_k}^{q_k})\\
 = \sum_{q_1+\ldots + q_k = d} \frac{1}{q_1!\cdots q_k!} \langle\prod_i\vartheta_{D_i}^{q_i}\rangle  = \sum_{-K_X \cdot A = d} n_\textbf{A}t^\textbf{A}\]
Multiplying both sides of the equality above by $d!$ and comparing with the inner most sum of Equation \ref{regqp} yields the desired result.

%

\end{proof}

Using the mirror theorem above, we derive integrality of regularized quantum periods for a large class of Fano variety. We recall the statement from the introduction:

\begin{corollary}
Assuming $X$ is a Fano such that there exists $D \in |-K_X|$ satisfying the conditions of Theorem \ref{mthm1}, then the regularized quantum periods are non-negative integers. 
\end{corollary}

\begin{proof}
Let $D = \sum_i D_i \in |-K_X|$ be a reduced representative, and consider a dominating compactification $(X',D') \rightarrow (X,D)$ of $X\setminus D$ with $(X',D')$ a log smooth pair. Then $(X',D')$ is a log Calabi-Yau pair, and by Theorem \ref{mthm1}, there exists an element $W_D = \sum_i \vartheta_{D_i} \in R_{(X',D')}$ whose classical period sequence equals the regularized quantum period sequence of $X$. By Lemma \ref{primnaive}, each classical period is a non-negative integer, and the conclusion follows.

\end{proof}

%
%

\section{Fano compactification of affine cluster varieties}

Suppose now that the coordinate ring for $U$ is is a skew symmetric affine cluster algebra, with $U$ compactified by a pair $(X,D)$, where $X$ is Fano and $D \in |-K_X|$, and $\mathcal{T} \subset U$ a seed torus with cocharacter lattice $N$. We say a Laurent polynomial $f \in \kk[N]$ is weakly mirror to a Fano variety $X$ if $f^d[z^0] = d!n_d$, as defined in the introduction. More generally, letting $J = (x_1,\ldots,x_n) \subset \kk[x_1,\ldots,x_n]$, we say $(f_k) \in \lim_{k \ge 0}\text{ }\kk[N][x_1,\ldots,x_n]/J^k$ is a formal weak mirror to $X$ if for any $d \ge 0$, the sequence in $k$, $(\overline{f}_k^d[z^0])$, is eventually the constant sequence $d!n_d$, where $\overline{f}_k$ is the Laurent polynomial given by taken the reduced expression for $f_k$ and setting $x_i = 1$ for all $i$.

We now prove existence of Laurent mirrors for Fano varieties of the above form:

\begin{corollary}\label{mirclust}
For $X$ as above, then $X$ has a formal Laurent polynomial mirror. If for any valuations $\nu_{D_i} \in B(U)$ there exists a seed $\mathcal{T} \subset U$ optimized for $\nu_{D_i}$ in the sense of \cite[Definition $9.1$]{bases_cluster}, then $X$ has a Laurent polynomial mirror $W_s$ associated with any seed $s$ of $U$.
\end{corollary}

\begin{proof}

By \cite[Theorem $D$]{clustermirr}, there exists an snc compactification $(X',D')$ of $U$ such that the canonical wall structure for the log Calabi-Yau pair $(X',D')$, after an appropriate base change, is the scattering diagram for the formal Fock-Goncharov dual cluster variety $\check{U}_{\operatorname{prin}}$. Moreover, after taking the associated rings of theta functions, this base change factors through the base change sending $z^{\textbf{A}}$ to $1$ for all $\textbf{A} \in \NE(X)$. The gluing construction of the mirror family over $\text{Spf }\kk[[x_1,\ldots,x_n]]$ implies the existence of formal tori $\hat{\mathcal{T}}_s = \text{Spf }\kk[[x_1,\ldots,x_n]][N] \rightarrow \text{Spf }R_{(X',D')}$, and by pulling back $W \in R_{(X',D')}$ along the inclusion, we produce a formal Laurent polynomial $W_s \in \kk[[x_1,\ldots,x_n]][N]$. Since the trace form of the mirror algebra is independent of the choice of compactification of $U$ after setting $z^{\textbf{A}} = 1$, Theorem \ref{mthm1} implies $W_s$ is a formal Laurent polynomial mirror for $X$. If $U$ has optimized seeds for each $\nu_{D_i}$, then by \cite[Lemma $9.3(1)$]{bases_cluster}, the formal theta functions $\vartheta_{D_i}$ are global monomials, hence restrict to Laurent polynomials in any formal torus chart. In particular, the formal Laurent polynomial $W_s$ is in fact a Laurent polynomial. Since the tori are in bijection with seeds of $U$, we have a Laurent polynomial mirror associated with every seed $s$ of $U$. 
\end{proof}

\begin{proof}[Proof of Corollary \ref{qdiff}]
By Corollary \ref{mirclust}, there exists a Laurent polynomial $f_D$ associated with the pair $(X, X\setminus U)$ such that $\pi_{f_D} = \hat{G}_X$. By Cauchy's residue theorem, we have:

\[g(t) = \frac{1}{(2\pi i)^n}\int_{\Gamma} e^{Wt} \bigwedge_j\frac{dx_j^*}{x_j^*} = \frac{1}{(2\pi i)^n}\sum_{d \ge 0} \int_{\Gamma} \frac{W^d}{d!}t^d \bigwedge_j\frac{dx_j^*}{x_j^*} = \sum_{d \ge 0} n_dt^{d}\]
The righthand term is the quantum period series of $X$, and satisfies the quantum differential equation for $X$ by construction.
\end{proof}

\begin{remark}
Tonkonog in \cite{Tonk19} proves an analogous result in symplectic geometry, where the Laurent polynomial $W_{D,s}$ is instead the Landau-Ginzburg potential associated with a monotone Lagrangian torus $L \subset X$. Under the expectation that $B$ is the base of the SYZ fibration for $X\setminus U$, the choice of Lagrangian torus $L$ corresponds to a fiber of the SYZ fibration, hence corresponds to a general point $p \in B$, and thus a choice of seed $s$ of $U$. 

\end{remark}

In order to ensure that $U$ has a Laurent polynomial mirror, we will assume that the optimized seed condition of Corollary \ref{mirclust} is satisfied. With this assumption, given any seed $\mathcal{T}_s^{\vee} \subset \check{U}$ of the Fock-Goncharov dual cluster variety, we have $W_D \in \Gamma(\check{U},\mathcal{O}_{\check{U}})$, and let $W_{D,s}$ be the restriction of $W_D \in R_{(X',D')}$ to $\mathcal{T}_s^{\vee}$. This gives a Laurent polynomial, with Newton polytope $\Delta_{W,s} \subset M_{\mathbb{R}}$, where $M_{\mathbb{R}}$ the real vector space associated with the lattice of characters on $\mathcal{T}_s^{\vee}$.


%
%

We now use the polar dual of the polytope $\Delta_{W,s}$, $\Delta_{W,s}^{\vee}$, to recover the compactification $X$ as well as a toric degeneration to the polarized toric variety associated with $\Delta_{W,s}^{\vee}$. To begin, note that $\Delta_{W,s}$ is a closed subset of the space of characters of the torus of $\mathcal{T}_s^{\vee} \subset \check{U}$, hence the polar dual $\Delta_{W,s}^{\vee}$ is contained in the space of cocharacters $N_{\mathbb{R}}$ of $\mathcal{T}_s^{\vee}$. Moreover, recall that $\Delta_{W,s}^{\vee}$ is given by the formula:

\begin{equation}\label{polytope}
\Delta_{W,s}^{\vee} = \{v \in N_{\mathbb{R}}\text{ }|\text{ }\operatorname{Trop}(W|_{\mathcal{T}_s^{\vee}})(v) \ge -1\} = \cap_{i=1}^k \{v \in N_{\mathbb{R}}\text{ }|\text{ }\operatorname{Trop}(\vartheta_{D_i}|_{\mathcal{T}_s^{\vee}}) \ge -1\}.
\end{equation}

In the previous formula, we identify $W$ and $\vartheta_{D_i}$ with Laurent polynomials, and $\operatorname{Trop}(f): N_\mathbb{R} \rightarrow \mathbb{R}$ the piecewise linear functions induced by a Laurent polynomial $f$ with integer coefficients, in which we replace the binary operations of multiplication with addition and addition with minimum. After identifying the space of cocharacters of $\mathcal{T}_s^{\vee}$ with the space of characters of $\mathcal{T}_s$, $\Delta_{W,s}^{\vee}$ is identified naturally with a convex rational polytope in the character lattice of the torus $\mathcal{T}_s$, hence encodes a projective toric variety $X_{W,s}^{\tor}$ which compactifies $\mathcal{T}_s$. If we assume in addition that the cluster variety $U$ satisfies the full Fock-Goncharov conjecture, we use results from \cite{bases_cluster} to show in that this polytope also encodes a toric degeneration of $X$ to $X_{W,s}^{tor}$. To do so, we recall the following class of polytope introduced in \cite{bases_cluster}:

\begin{definition}
A closed subset $\Xi \subset V(\mathbb{R}^T)$ is \emph{positive} if for any non-negative integers $d_1,d_2$, any $p_1 \in d_1\Xi(\NN)$, $p_2 \in d_2\Xi(\NN)$ and any $r \in V(\NN^T)$ with $N_{p_1,p_2}^r \not= 0$, we have $r \in (d_1+d_2)\Xi(\NN)$. 
\end{definition}
Positivity of $\Delta_{W,s}$ will follow from the following lemma:

%

\begin{lemma}\label{thetabasis}
Suppose that $U$ satisfies the full Fock-Goncharov conjecture and for every divisorial valuation $\nu_{D_i}$ corresponding to a component $D_i \subset D$ there is an optimized seed for $\nu_{D_i}$. Then $H^0(X,\mathcal{O}_X(rD))$ is a vector space with $\kk$ basis $\{\vartheta_p\text{ }|\text{ }p \in r\Delta_{W,s}^{\vee}\}$, where $r\Delta_{W,s}^{\vee}$ is the $r^{th}$ dilation of the polytope $\Delta_{W,s}^{\vee}$.
\end{lemma}


\begin{proof}

Since $U$ satisfies the full Fock-Goncharov conjecture, $\kk[U]$ has a canonical basis of theta functions indexed by $\check{U}^{\trop}(\ZZ)$. By the assumption that $U$ has optimized seeds for every $\nu_{D_i}$, by \cite[Lemma $9.3(1)$]{bases_cluster}, $\check{U}_{\operatorname{prin}}$ contains a dense formal torus such that the theta function $\vartheta_{D_i}$ restricts to a standard monomial. Moreover, observe that $\vartheta_p \in H^0(X,\mathcal{O}_X(rD))$ if and only if $\nu_{D_i}(\vartheta_{p}) \ge -r$ for all $i$. By \cite[Lemma $9.10(3)$]{bases_cluster}, the previous inequalities hold if and only if for all components $D_i \subset D$ and $\vartheta_{D_i} \in \kk[\check{U}]$ the element of the canonical basis on the dual cluster variety associated with $\nu_{D_i} \in B(\NN)$:
\[p(\vartheta_{D_i}) \ge -r.\]
We may compute the valuations appearing on the left side of the inequality above by restricting $\vartheta_{D_i}$ to a cluster chart and evaluating $p$, thought of as a cocharacter on the cluster chart, with the tropicalization of the resulting Laurent polynomial $\vartheta_{D_i,s}$. Indeed, $\vartheta_{D_i}$ and $\vartheta_{D_i,s}$ induce the same functions on the essential skeleton of the torus $\mathcal{T}_s^{\vee} \subset \check{U}$, and the latter function is given by tropicalization. It follows by definition of $\Delta_{W,s}^{\vee}$ that the latter inequality holds if and only if $p \in r\Delta_{W,s}^{\vee}$. Hence, the integral points of $r\Delta_{W,s}^{\vee}$ are in bijection with theta functions $\vartheta_p$ such that $\vartheta_p \in H^0(X,\mathcal{O}_X(rD))$.

Now suppose $f \in H^0(X,\mathcal{O}_X(rD))$. In particular, we must have $f$ is regular on $U$, hence we can write $f = \sum_i a_i\vartheta_{p_i}$ for some $a_i \in \kk$. Since we have established that $\vartheta_p \in H^0(X,\mathcal{O}(rD))$ if and only if $p \in r\Delta_{W,s}^{\vee}$, we must have $p_i \in r\Delta_{W,s}^{\vee}$ for all $a_i \not= 0$. The final claim that the theta functions form a basis follows from that fact that theta function form a basis for $\kk[U]$. 
\end{proof}

\begin{proposition}\label{positive}
$\Delta_{W,s}^{\vee}$ is a full dimensional compact positive polytope. Moreover, there is a corresponding toric degeneration of $X$ with special fiber $X_{W,s}^{tor}$. 
\end{proposition}

\begin{proof}

%
%
%

To prove $\Delta_{W,s}^{\vee}$ is a full dimensional compact positive polytope, we note that by Lemma \ref{thetabasis}, the theta functions associated with integral points of $r\Delta_{W,s}^{\vee}$ give a basis for $H^0(X,\mathcal{O}_X(rD))$ for $r\ge 0$. Since every theta function is contained in $H^0(X,\mathcal{O}_X(rD))$ for some $r\ge 0$ and $H^0(X,\mathcal{O}_X(rD))$ is a finite dimensional vector space, the dilates of the rational polytope $\Delta_{W,s}^{\vee}$ cover $N_{\mathbb{R}}$ and $\Delta_{W,s}^{\vee}$ must be closed and bounded, hence $\Delta^{\vee}_{W,s} \subset M_{\mathbb{R}}$ is a compact and full dimensional polytope. Moreover, the algebra $\mathbb{C}[U]$ is a filtered algebra, with associated graded vector space having homogeneous basis $\vartheta_p$ with  $\deg(\vartheta_p) = min\{r\text{ }|\text{ }p \in r\Delta_{W,s}^{\vee}\}$. In particular, we find $\Delta_{W,s}^{\vee}$ is also a positive polytope. 

The construction of the toric degeneration with special fiber $X_{\Delta_{W,s}^{\vee}}$ follows from \cite[Theorem $8.39$]{bases_cluster}. Indeed, the general fiber of the construction from loc. cit. is given by the proj construction of the graded ring $\oplus_{k\ge 0} \oplus_{p \in k\Delta_{W,s}^{\vee}}\vartheta_pT^k$, and this graded ring is precisely the Rees algebra associated with the filtered ring structure on $\Gamma(U,\mathcal{O}_U)$ imposed by the ample divisor $D \subset X$ by Lemma \ref{thetabasis}. Thus, the general fiber is $X = \text{Proj}\oplus_{r\ge 0} H^0(X,\mathcal{O}_X(rD))$. The special fiber is the projective toric variety associated to $\Delta_{W,s}^{\vee}$, which is $X_{W,s}^{\tor}$ by definition. 

\end{proof}

%
%
%
%
%

Finally, we observe that for $X$ as above now assumed to be an $\mathcal{X}$ cluster variety, the polytopes $\Delta_{W,s}^{\vee}$ constructed above are Newton-Okounkov bodies for the anticanonical line bundle:

\begin{theorem}
For a valuation $\val: K(X)^{\times} \rightarrow \mathbb{R}^n$, and a line bundle $L \in \Pic(X)$, recall the associated Newton-Okounkov body: 
\[NO_{\val}(L) = \overline{\operatorname{ConvexHull}(\bigcup_{r=1}^{\infty} \frac{1}{r} \val(H^0(X,L^{\otimes r})))}\]
For $X$ a $\mathcal{X}$-cluster variety satisfying the full Fock-Goncharov conjecture, for every seed $T^{\vee}_s\subset \check{U}$, there is an associated valuation $\val_s: K(X)^\times \rightarrow \ZZ^n$ such that $NO_{\val_s}(-K_X) = \Delta_{W,s}^{\vee}$. 
\end{theorem}
\begin{proof}
For each seed torus $\mathcal{T}_s \subset U$ giving in particular an isomorphism $M_{\mathcal{T}_s} \cong \ZZ^n$, there is a valuation $\val_s: K(X) \rightarrow \ZZ^n \cong M_{\mathcal{T}_s}$, defined on the theta basis by $\nu(\vartheta_p) = \phi(p)$ for $\phi: \check{U}^{trop}(\ZZ) \rightarrow N_{\mathcal{T}^{\vee}_s} = M_{\mathcal{T}_s}$ the piecewise linear isomorphism determined by inclusion $\mathcal{T}^\vee_s \subset \check{U}$. To see this assignment extends to a valuation, we note that the valuation considered by Rietsch and Williams in \cite[Definition $8.1$]{NOmirr} , which can be defined for a seed of an arbitrary $\mathcal{X}$ cluster variety satisfying the full Fock-Goncharov conjecture, and satisfies the equality $\val_s(\vartheta_p) = \phi(p)$ by Theorem $16.15$ of loc. cit.

%
 
 By Lemma \ref{thetabasis}, the integral points of $r\Delta_{W,s}^{\vee}$ determine a basis of sections for $-K_X^{\otimes r}$. The rational polytope $\Delta_{W,s}^{\vee}$ may not necessarily be integral, but there exists some $r \ge 1$ such that $r\Delta_{W,s}^{\vee}$ is an integral polytope. In particular, $r\Delta_{W,s}^{\vee}$ is the convex hull of its integral points. It now follows from the definition of Newton-Okounkov bodies that $NO_{\val_s}(-K_X) = \Delta_{W,s}^{\vee}$. 
\end{proof}

\begin{remark}
We note that Bossinger, Cheung, Magee and N\'ajera Ch\'avez in \cite{minmodclust} have investigated Newton-Okounkov bodies for partial compactifications of type $\mathcal{A}$ cluster varieties. The authors of op. cit. give an expression for Newton-Okounkov bodies in terms of the broken line convex hull of integral points in $B(\check{U})$ associated with theta function in $H^0(X,L)$ for $D$ an effective ample divisor with support in the complement of $U$. A choice of seed $\mathcal{T}_s \subset U$ then gives an identification of $B(\check{U})(\NN)$ with $\mathbb{Z}^n$.

Our work is closer in spirit to \cite{NOmirr}, which identifies families of Newton-Okounkov bodies of type $\mathcal{X}$ cluster varieties as polar dual to Newton polytope of a mirror Laurent polynomial constructed using the superpotential on the mirror. In particular, the mutations of Newton-Okounkov bodies induced by wall crossings follow from the analogous mutations of the mirror Laurent polynomial induced by wall crossing. This perspective also makes clear that the Newton-Okounkov body is determined by a finite number of inequalities, and highlights its relation to the Gromov-Witten theory of $X$.

Finally, while we worked with the ample anticanonical divisor in this section, similar statements hold if we replaced the anticanonical with another ample divisor supported on $D$, by a modification of the description of the polytope given in Equation \ref{polytope} similar to \cite[Definition $10.14$]{NOmirr}. An enumerative interpretation still follows from Theorems \ref{enum} and \ref{logTRR}, but the relationship with quantum periods only holds when using the anticanonical divisor. 
\end{remark}

\section{Intrinsic mirror to Grassmanian}


A series of papers written by Marsh and Rietsch in \cite{LGgr} and Rietsch and Williams in \cite{NOmirr} express the mirror superpotentials to the Grassmanian $X = \Gr(n-k,n)$ in terms of Pl\"ucker coordinates on the dual Grassmanian $\Gr(k,n)$, and show that these LG mirrors encode Newton-Okounkov bodies and toric degenerations of $X$. We will show in this section that the intrinsic mirror construction above recovers the Pl\"ucker mirror superpotential considered in these papers. 



We first recall the positroid stratification of $X = \Gr(n-k,n)$, with dense cell given by the open positroid variety $X^\circ \subset X$ which is a log Calabi-Yau. Moreover, the coordinate ring for $X^\circ$ has a skew-symmetric $\mathcal{A}$ cluster structure, first investigated by Scott in \cite{scottclust}. It follows from \cite[Theorem $21.27$]{archmirror} that away from codimension $2$, the Fock-Goncharov dual $\mathcal{X}$ cluster variety $\check{X}^\circ$ is the general fiber of any mirror family to $X^\circ$. It is known that the coordinate ring of the Fock-Goncharov dual cluster variety is the coordinate ring of the open positroid variety $\check{X}^\circ \subset \check{X} = \Gr(k,n)$. We delay a proof of this fact until Lemma \ref{dualclust} below.


To describe the compactifying divisor $D = X \setminus X^\circ$, recall that $X$ has a homogeneous coordinate ring whose homogeneous monomials are the Pl\"ucker coordinates. These coordinates are indexed by sets $J \in \binom{[n]}{n-k}$, by taking the associated $(n-k)\times (n-k)$ minor of a matrix representing a point of $X$. It is typical to further identify the set of Pl\"ucker coordinates with the set $\mathcal{P}_{k,n}$ of Young diagram in a $(n-k)\times k$ box. We denote by $P_\mu$ the Pl\"ucker coordinate corresponding to a Young diagram $\mu$. There is a natural bijection between Young diagrams in $\mathcal{P}_{k,n}$ and subsets of $[n]$ of size $n-k$, given by sending a Young diagram to the collection of southward steps in the path along the lower border of the diagram starting at the northeast corner of the diagram and ending in the southwest corner. Moreover, the components of $D$ are given by the vanishing of $n$ Pl\"ucker coordinates $P_{\mu_i}$ with $\mu_i$ a rectangular Young diagram with \emph{westward} steps given by the length $k$ interval $[i+1,i+k]\subset [n]$ equipped with its cyclic order. We note that that $\mu_{n-k}$ is a maximal rectangle in this collection of dimension $(n-k) \times k$. 

Additionally we note that $\mathcal{P}_{k,n}$ also indexes Pl\"ucker coordinates on the dual Grassmannian $\check{X}$, by associating to a Young diagram $\mu$ the subset $J_{\mu} \in \binom{[n]}{k}$ given by taking the collection of westward steps in the path along the lower border of $\mu$ starting in the northeast corner and ending in the southwest corner. The Pl\"ucker coordinates on the dual Grassmannian will be denoted by $p_\mu$ for $\mu \in \mathcal{P}_{k,n}$. 

In general, the divisor $D$ is not normal crossings, so $(X,D)$ is not a log Calabi-Yau pair. However, there exists a birational model $(Y,D')$ which is an isomorphism on $X^\circ$ which is a log Calabi-Yau pair. Indeed, away from codimension $2$, $X\setminus D$ is a cluster variety, hence in particular posses a unique non-vanishing holomorphc volume form which restricts to the standard volume form on a dense algebraic torus. It follows from this that $U$ is log Calabi-Yau. By passing to a resolution of singularities of $D\subset X$, we find a simple normal crossings compactification of $X\setminus D$, $(Y,D')$, which is a log Calabi-Yau pair. 

By Theorem \ref{mthm1}, we know there exists an element $W_D \in R_{(Y,D')}$ which is weakly mirror to the compactification $X^\circ\subset X$, and can be expressed in terms of the theta function basis for $R_{(Y,D')}$. In order to relate $W_D$ to a function on the dual Grassmanian, we consider seed tori of the cluster varieties on either side of the mirror. 

We first recall the $\mathcal{A}$ cluster structure on the coordinate ring of $X^\circ$. A collection of seeds for the $\mathcal{A}$ cluster structure are described by \emph{plabic graph} $G$, certain bicolored graphs embedded in a disk, see \cite[Section $3$]{NOmirr} for details. To every plabic graph $G$, there is a set of Young diagram $\widetilde{\mathcal{A}Coord}(G)(X)$ which are in bijection with monomials for the associated torus inclusion $\mathcal{T}_G \rightarrow X^\circ$. A special property of the $\mu_i$ associated with components of $D_i$ is that $\mu_i \in \widetilde{\mathcal{A}Coord}(G)(X)$ for all plabic graphs $G$. For a given plabic graph $G$, we fix an enumeration of $\widetilde{\mathcal{A}Coord}(G)(X)$ with $\mu_j$ for $1\le j \le n$ given by the rectangular diagram indexing the components of the divisor $D$. The coordinates for $\mathcal{T}_G$ are given by:

\[\mathcal{A}Coord(G)(X) = \{\frac{P_{\mu}}{P_{\mu_{n-k}}}\text{ }\mid\text{ }\mu \in \widetilde{\mathcal{A}Coord}(G)(X)\setminus \{\mu_{n-k}\}\}.\]
Analogously, the plabic graph $G$ also encodes an $\mathcal{A}$ cluster chart for $\check{X}^\circ$, with coordinates on $\mathcal{T}_G$ given by:
\[\mathcal{A}Coord(G)(\check{X}) = \{\frac{p_{\mu}}{p_{\emptyset}}\text{ }\mid\text{ }\mu \in \widetilde{\mathcal{A}Coord}(G)(\check{X})\setminus \{\mu_{\emptyset}\}\}.\]

Using the torus chart $\mathcal{T}_{G} \subset X^\circ$, we can identify the integral points of the essential skeleton of $X^\circ$ with the cocharacter lattice of the torus $\mathcal{T}_{G}$. Under this identification, the theta functions of the mirror are indexed by cocharacters of $\mathcal{T}_{G}$. Moreover, the superpotential $W_D$ can be expressed as a sum of theta functions $W_D = \sum_i \vartheta_{D_i}$, with $\vartheta_{D_i}$ the theta function corresponding to the divisorial valuation on $\kk(Y) = \kk(X)$ associated with the divisor $D_i$. In particular, denoting by $\nu_{\mu_i}$ the divisorial valuation associated with each divisor $D_{\mu_i}$ for $1\le i \le n$, restriction to the torus $\mathcal{T}_{G}$ induces a cocharacter determined by the following formula:

\begin{equation}\label{value}
\nu_{\mu_i}|_{\mathcal{T}_{G}}(\frac{P_{\mu_j}}{P_{\mu_{n-k}}}) = \delta_{i,j} - \delta_{i,n-k}.
\end{equation}

In the interest of working with the mirror to $X^\circ$, we now recall the proof which identifies the mirror with an analogous affine open of the Langlands dual variety $\Gr(n-k,n)$. 

\begin{lemma}\label{dualclust}
The open positroid varieties $X^\circ\subset \Gr(n-k,n)$ and $\check{X}^\circ \subset \Gr(k,n)$ are the affinizations of Fock-Goncharov dual cluster varieties of type $\mathcal{A}$ and type $\mathcal{X}$ respectively.
\end{lemma}

\begin{proof}
As described in sections $5$ and $6$ of \cite{NOmirr}, the open positroid varieties $X^\circ$ and $\check{X}^\circ$ after removing a codimension $2$ subschemes have $\mathcal{A}$ and $\mathcal{X}$ cluster structures with initial seeds $\mathcal{T}_{rec}\subset X^\circ,\mathcal{T}_{rec}^\vee\subset \check{X}^\circ$ respectively, and the same underlying quiver $Q(G_{rec})$. In particular, the ring of global functions on the cluster varieties are equal to the ring of global functions on the dual open positroid varieties. Since the open positroid varieties are affine, the lemma follows by definition of Fock-Goncharov dual cluster varieties, i.e. \cite[Definition $\text{A}.4$]{bases_cluster}.
\end{proof}

The dual torus $\mathcal{T}_{G}^\vee$ gives a corresponding seed of the $\mathcal{X}$ cluster structure on $\check{X}^\circ$, which is typically referred to as a network chart. To express the theta function mirror $W_D$ of $X$ in terms of Pl\"ucker coordinates on $\check{X}$, we will express each theta function $\vartheta_{D_i}$ contributing a term to $W_D$ as a rational monomial of Pl\"ucker coordinates. First, for $\val_G$ the valuation defined by Rietsch and Williams in \cite{NOmirr}, \cite[Theorem $16.15$]{NOmirr} shows $\val_G(\vartheta_m) = \val_G(\vartheta_{m'})$ if and only if $m = m'$. Moreover, Equation \ref{value} determines how the cocharacter $\val_G(\vartheta_{\mu_j})$ pairs with all characters on $\mathcal{T}_G$, hence uniquely determines $\mu_j$. Hence to show a monomial in Pl\"ucker coordinates $P$ is equal to $\vartheta_{\mu_j}$ as a regular function on $\check{X}^\circ$, it suffices to verify that the monomial is a theta function, and after restricting to the seed torus $\mathcal{T}^{\vee}_{G}$ for some plabic graph $G$, we have $\val_G(P)$ is given by Equation \ref{value}.



To construct monomials satisfying the first condition, we note that by \cite[Theorem $11.1$]{LGgr}, for any Pl\"ucker coordinate $p_\mu$ on $\check{X}$, there exits a plabic graph $G$ such that $\mu \in \widetilde{\mathcal{A}Coord}(G)(\check{X})$, hence $\frac{p_\mu}{p_{\emptyset}}$ restricts to a monomial on a seed torus for the $\mathcal{A}$ cluster structure on $\check{X}$. Since $\frac{p_{\mu_i}}{p_{\mu_{\emptyset}}} \in \mathcal{A}Coord(G)(\check{U})$ for all plabic graphs $G$, we have $\frac{p_\mu}{p_{\mu_i}}$ restricts to a Laurent monomial on $\mathcal{T}_{G}^{\vee} \subset \check{U}$. Since $\frac{p_\mu}{p_{\mu_i}}$ is regular on $\check{U}$ and every regular function on $\check{U}$ which restricts to a Laurent monomial on a seed torus is a theta function on $\check{X}^\circ$, all of the ratios of Pl\"ucker coordinates above are theta functions. 

To pick a particular diagram relevant for finding a Pl\"ucker expression for $\vartheta_{D_i}$, for $\mu_i$ a Young diagram with westwards steps given by the interval $[i+1,i+k]$ in $[n]$ interpreted cyclically, we define a new Young diagram $\mu_i^\Box$ with westward steps given by:
 $$ [i+1,i+k-1]\cup \{i+k+1\}.$$
  For $\mu_i$ not a full box, $\mu_i^{\Box}$ is the unique Young diagram such that $\mu_i^{\Box}\setminus \mu_i$ consists of a single box. When $\mu_i$ is a full box, $\mu_i^{\Box}$ is the complement in $\mu_i$ of the outer rim.

 We also will require a slight variant \cite[Lemma $19.3$]{NOmirr}. We first recall a definition and a theorem from \cite{NOmirr} which gives an expression for the valuations of Pl\"ucker coordinates:

\begin{definition}[Definition 14.3 \cite{NOmirr}]
Given two Young diagrams $\lambda,\mu$ contained in a  $(n-k)\times k$ rectangle, let $\mu\setminus \lambda$ denote the corresponding skew diagram, i.e. the set of boxes remaining if we justify both $\mu$ and $\lambda$ at the top left of a $(n-k)\times k$ rectangle and removing boxes of $\mu$ which are contained in $\lambda$. Let $\text{MaxDiag}(\mu\setminus \lambda)$ be the maximum number of boxes contained in an antidiagonal of the rectangle.
\end{definition}

\begin{lemma}[Theorem 15.1 \cite{NOmirr}]\label{mdiag}
Let $G$ be any reduced plabic graph associated with a seed of the $\mathcal{X}$ cluster structure on $\kk[\check{X}^\circ]$ and $\lambda$ a Young diagram in a $(n-k)\times k$ rectangle. Then for any $\mu \in \mathcal{P}_G$:
\[\val_G(\frac{p_\lambda}{p_{\emptyset}})_\mu = \operatorname{MaxDiag}(\lambda\setminus \mu).\]
\end{lemma}

\begin{proof}
The Pl\"ucker coordinates $p_{\lambda}$ on $\Gr(k,n)$ indexed by Young diagrams in $\mathcal{P}_{k,n}$ are also indexed by Young diagrams in $\mathcal{P}_{n-k,n}$. We refer to a Pl\"ucker coordinate associated with a Young diagram $\lambda \in \mathcal{P}_{n-k,n}$ by $p'_{\lambda}$. In terms  of these coordinates, we have by \cite[Theorem $15.1$]{NOmirr} that for $\lambda,\mu \in \mathcal{P}_{n-k,n}$:
\[\val_G(\frac{p_\lambda'}{p_{\text{max}}'})_\mu = \text{MaxDiag}(\mu\setminus \lambda),\]
where $\mu_{\text{max}}\in \mathcal{P}_{n-k,n}$ is the filled $k \times (n-k)$ rectangle. The bijection $\sigma: \mathcal{P}_{n-k,n} \rightarrow \mathcal{P}_{k,n}$ induced by $p'_{\lambda} = p_{\sigma(\lambda)}$ is given by reflecting $(\mu_{\max}\setminus \lambda)$ across the antidiagonal. Indeed, this operation sends a Young diagram with lower boundary containing southward path segments given by a subset $J \subset [n]$ to the Young diagram with lower boundary containing westward path segments in the subset $J \subset [n]$. We now observe that:

\[\text{MaxDiag}(\mu\setminus\lambda) = \text{MaxDiag}(\sigma(\lambda)\setminus\sigma(\mu)).\]
Since $p'_{\mu_{\text{max}}} = p_{\sigma(\mu_{\text{max}})} = p_{\emptyset} $, the conclusion follows from the above equation.

\end{proof}

The following Lemma establishes that a certain ratio of Pl\"ucker coordinates has valuation matching $\vartheta_{D_i}$:



\begin{lemma}\label{valgrass}
Fix $i \in \{0,1,\ldots,n-1\}$ and let $G$ be a plabic graph and $e^{(j)} = \frac{P_{\mu_j}}{P_{\mu_{n-k}}}|_{\mathcal{T}_G}$ with $\mu_j \in \widetilde{\mathcal{A}Coord}(G)(X)$ a character on the seed torus $\mathcal{T}_{G}$ for the $\mathcal{A}$ cluster structure on $X^\circ$, hence a cocharacter on the seed torus $\mathcal{T}_{G}^{\vee}$ for the $\mathcal{X}$ cluster structure on $\check{X}^\circ$. Then we have:
\begin{equation}\label{thetaval}
\val_G(\frac{p_{\mu_i^\Box}}{p_{\mu_i}})(e^{(j)}) = \delta_{i,j} - \delta_{i,n-k} = \val_G(\vartheta_{D_i})(e^{(j)}).
\end{equation}
\end{lemma}

\begin{proof}
By Lemma \ref{mdiag}, we can compute:
\[\val_G(\frac{p_{\mu_i}^{\Box}}{p_{\mu_i}})(e^{(j)}) = \text{MaxDiag}(\mu_i^{\Box}\setminus \mu_j)-\text{MaxDiag}(\mu_i\setminus \mu_j).\]
Thus, to compute $\val_G(\frac{p_{\mu_i^\Box}}{p_{\mu_i}})(e^{(j)})$, we must consider when $\text{MaxDiag}(\mu_i^{\Box}\setminus \mu_j)$ differs from $\text{MaxDiag}(\mu_i\setminus \mu_j)$. 

We first suppose $i \not= n-k$, hence $\mu_i$ is not a filled $(n-k)\times k$ rectangle and $\mu_i^{\Box}\setminus \mu_i$ contains a single box. Since $\mu_i \subset \mu_i^{\Box}$, we have:
\[\mu_i \setminus \mu_j \subset \mu_i^{\Box}\setminus \mu_j.\]
In particular, we have $\text{MaxDiag}(\mu_i \setminus \mu_j) \le \text{MaxDiag}(\mu_i^{\Box}\setminus \mu_j)$. Since the additional box of $\mu_i^{\Box}$ is not contained in the same antidiagonal as any other box of $\mu_i^{\Box}$, we have: 
 \[\text{MaxDiag}(\mu_i^{\Box}\setminus \mu_j) \le \text{max}\{1,\text{MaxDiag}(\mu_i\setminus \mu_j)\}.\]
Thus, to have $\text{MaxDiag}(\mu_i^{\Box} \setminus \mu_j) > \text{MaxDiag}(\mu_i\setminus \mu_j)$, we must have $\text{MaxDiag}(\mu_i\setminus \mu_j) = 0$. Assuming $\text{MaxDiag}(\mu_i^{\Box}\setminus \mu_j) \not= 0$, this can only happen if $i = j$. This verifies the formula when $i \not= k$. 

When $i = n-k$, we have $\mu_{i}$ is a filled $(n-k)\times k$ rectangle, and $\mu_{i}^{\Box}$ is the sub diagram with a box removed from each antidiagonal from the right and bottom edges. It follows that for any diagram $\mu_j$, $\text{MaxDiag}(\mu_i\setminus \mu_j) = 1+\text{MaxDiag}(\mu_i^{\Box}\setminus \mu_j)$ when $j \not= i$, and $0 = \text{MaxDiag}(\mu_i\setminus \mu_j) = \text{MaxDiag}(\mu_i^{\Box}\setminus \mu_j)$ when $i = j$. The left most equality of Equation \ref{thetaval} now follows. The rightmost equality of Equation \ref{thetaval} follows from $\val_G(\vartheta_{D_i}) = \phi(\rho_{D_i})$ for $\rho_{D_i} \in \Sigma(X)(\NN)$ the primitive point associated with the divisorial valuation of $D_i$ and Equation \ref{value}.
\end{proof}

As a result, we find $\val_G(\vartheta_{D_i}) = \val_G(\frac{p_{\mu_i^\Box}}{p_{\mu_i}})$, hence $\vartheta_{D_i} = \frac{p_{\mu_i^\Box}}{p_{\mu_i}}$ and $W_D = \sum_i \frac{p_{\mu_i^\Box}}{p_{\mu_i}} \in \kk[\check{X}^\circ]$, recovering the Pl\"ucker superpotential found by Marsh and Rietsch in \cite{LGgr} after setting $q =1$ for $q$ the quantum parameter of the superpotential of op. cit. Since the Fock-Goncharov dual cluster algebra $\kk[\check{X}^\circ]$ is the ring of functions on the general fiber of the mirror to $X^\circ$, Theorem \ref{mthm1} implies the classical periods of $W_D$ give the regularized quantum periods after setting all Novikov parameters $t^{\textbf{A}}\in \kk[H_2^+(X)]$ to $1$.

To recover the LG superpotential $W_q \in \kk[\check{X}][q]$ with Novikov parameters $q^d$ keeping track of degree $q$ with respect to the Pl\"ucker polarization, note that since the Picard rank of $X$ is $1$, and all divisors $D_i$ represent the same primitive element of $\Pic(X)$, the only non-trivial classical periods come from products $\prod_i\vartheta_{D_i}^l$ for some $l \ge 0$. Moreover, Marsh and Rietsch identify the restriction of $W_D$ to a certain seed torus with a Laurent mirror to $X$, $W_{s} \in \kk[z_1^{\pm 1},\ldots,z_{k(n-k)}^{\pm 1}]$, originally conjectured by Eguchi, Hori and Xiong, which satisfies $W_s^n[z^{\vec{0}}] \not= 0$. It follows that $W_{D,q}^n[\vartheta_0] = n_1q \not= 0$ for any lift $W_{q} \in \kk[\check{X}^\circ][q]$ of $W_{D}\in \kk[\check{X}^\circ]$ whose classical periods recover the regularized quantum periods without fixing Novikov parameters. This can only happen if $W_{q}$ contains exactly one term of the form $q\vartheta_{D_i}$, and all other terms of $W_D$ are decorated with the trivial curve class. Up to a global automorphism of $\check{X}$ which cyclically permutes the components of $\check{D}\subset \check{X}$, we recover the LG superpotential of \cite{LGgr}.

Additionally, since the restriction of these theta functions to a $\mathcal{A}$-cluster seed torus in $\check{X}^\circ$ are monomials, the divisorial valuations $\nu_{D_i}$ all have optimized seeds, and the results of Section $6$ apply. In particular, we recover the Newton-Okounkov bodies and toric degenerations of $X$ constructed by Rietsch and Williams in \cite{NOmirr}.



Finally, we recall that the Laurent polynomial $p_J^G$ given by restricting the Pl\"ucker coordinate $p_J$ to an $\mathcal{X}$-cluster chart associated with a plabic graph $G$ is referred to as flow polynomial. The coefficients of this polynomial have an enumerative description in terms of certain perfect matchings on the graph $G$. We observe in the following proposition an additional enumerative interpretation for the coefficients of $p_J^{G}$ in terms of non-archimedean disk counts in the Berkovich analytic space $(\Gr(n-k,n)^{\circ})^{\an}$ over the trivially valued field $k$:

\begin{proposition}
For any choice of plabic graph $G$ and $\vartheta_{D_i} = \frac{p_{\mu_i^{\Box}}}{p_{\mu_i}}$ a summand in the mirror Pl\"ucker superpotential to $\Gr(n-k,n)$, there exists $x \in M_{\mathbb{R}}$ such that the coefficients of the flow polynomial $\frac{p_{\mu_i^{\Box}}^G}{p_{\mu_i}} \in \kk[\mathcal{T}^{\vee}_G]$ are counts of non-archimedean cylinders in $(X^\circ)^{\an}$ defined in \cite{archmirror}:
\[\frac{p_{\mu_i^{\Box}}^G}{p_{\mu_i}^G}  = \sum_{e \in M,S \in SP_{x,\nu_{\mu_i},e},A \in \NE(Y)} N(S,A)z^e,\]
with $SP_{x,\nu_{\mu_i},e}$ a set of potential spines of non-archimedean cylinders as defined in Definition $20.20$ of op. cit. and $N(S,Y)$ the count of non-archimedean cylinders in $(\Gr(n-k,n)^{\circ})^{\an}$ with spine $S$.
\end{proposition}

\begin{proof}
By \cite[Theorem $6.8$]{NOmirr}, the flow polynomial $p_{\mu_i^{\Box}}^G$ is given by the restriction of the regular function $\frac{p_{\mu_i^{\Box}}}{p_{\emptyset}} \in \kk[\check{X}^\circ]$ to a network torus $\Phi_G: \mathcal{T}_G \rightarrow \check{X}^\circ$ associated with a plabic graph $G$. Since the regular functions $\vartheta_{D_i} = \frac{p_{\mu_i^{\Box}}}{p_{\mu_i}}$ are part of the canonical basis as previously established, \cite[Theorem $6.2$]{bases_cluster} states that there exists a point $x \in N_{\mathbb{R}}$ such that the restriction of these functions to the seed torus $\Phi_G$ is given by a sum of final monomials of broken lines in the scattering diagram from $\check{X}^\circ$ with initial direction given by the cocharacter $\nu_{\mu_i}$ in the seed torus $\mathcal{T}_G$ and final position given by $x$. Moreover, by \cite[Theorem $21.27$]{archmirror}, the scattering diagram whose chambers encode the seeds of the cluster variety $\check{X}^\circ$ is also the scattering diagram constructed in loc cit. coming from counts of infinitesimal analytic cylinders in $(\Gr(n-k,n)^{\circ})^{\an}$. In particular, the broken line expansion $\vartheta_{x,m}$ considered above is also the local theta function defined in Definition $20.20$ after setting all curve classes equal to $0$, yielding the required expression.

\end{proof}

%

\begin{remark}
In the symplectic setting, Castranovo proves in \cite[Proposition $3.5$]{hmsgrass} that the Laurent polynomial $W_D|_{\mathcal{T}_{rec}}$ is a generating function of Maslov index $2$ disk counts in the log Calabi-Yau $X^\circ$, with $\mathcal{T}_{rec}$ the seed torus of $\Gr(k,n)$ with cluster monomials indexed by rectangular Young diagrams. By combining the two propositions above, we derive an equality between the counts of non-archimedean disks above and counts of Maslov index $2$ disks. The Laurent polynomials also have a combinatorial interpretation in terms of counts of perfect matchings. It would be interesting to upgrade the equality of numbers to an explicit bijection between perfect matchings and non-archimedean disks.
 
In order to generalize the argument above to produce Laurent mirrors to other homogeneous flag varieties, we need a similar explicit understanding of both the cluster structure of $X^\circ$ and the valuations of theta functions on the mirror. In particular, recent work of Spacek and Wang in \cite{LGcomin} construct type independent Pl\"ucker mirrors for cominiscule homogeneous varieties $X = G/P$. It would be interesting if the arguments of this section could be generalized to hold in this general setting.
\end{remark}

\section{Frobenius Structure conjecture for smooth pairs $(X,D)$}

We now consider the opposing situation to what we considered above, namely let $X$ be a Fano and $D \in |-K_X|$ be a smooth representative. Then we may construct an algebra $R_{(X,D)}$ associated with this pair, which is equipped with a theta function basis $\vartheta_n$ for $n \in \Sigma(X,D)(\ZZ) = \NN$. In this case, the algebra is always isomorphic to $\kk[\vartheta_1]$, with $\vartheta_1$ the theta function associated with the primitive lattice point of $\Sigma(X,D)$. However, the theta basis is different from the power basis, and structure constants in the former basis encode non-trivial log Gromov-Witten invariants of the pair $(X,D)$. Moreover, by Theorem \ref{mthm1}, we have $\vartheta_1^k[\vartheta_0] = k!N_k$. We wish to show that these structure constants uniquely determine all remaining structure constants in the theta function basis of $R_{(X,D)}$. In particular, the quantum periods of $X$ completely determine the theta basis for the mirror algebra $R_{(X,D)}$.

To show this, we consider a third class of log Gromov-Witten invariants:

\begin{definition}\label{2pinv}
For $p,q \in \NN$ and $\textbf{A} \in \NE(X)$, consider the moduli space $\mathscr{M}(X,\textbf{A},(p,q))$ of $2$-pointed log stable maps of curve class $\textbf{A}$ and contact orders with respect to $D$ given by $p$ and $q$ at marked points $x_p$ and $x_q$ respectively. Define the $2$-pointed log Gromov-Witten invariant:
\[N_{p,q}^\textbf{A} := \int_{[\mathscr{M}(X,\textbf{A},(p,q))]^{\vir}} \ev_{x_q}^*([\pt]),\]
and the element $N_{p,q} \in \mathbb{Q}[\NE(X)]$ to be the generating series $N_{p,q} = \sum_{\textbf{A} \in \NE(X)} N_{p,q}^\textbf{A}z^{\textbf{A}}.$

\end{definition}

These invariants were previously considered in \cite{2pointG}, \cite{Ysmooth}, \cite{GRZ1} and \cite{YouPLG}. In particular, these works give the following relationship between the theta function structure constants for $R_{(X,D)}$ and the invariants $N_{p,q}$:

\begin{theorem}\label{stconrel}
Let $(X,D)$ be a log Calabi-Yau pair with $D$ smooth and $p,q \in \Sigma(X)$, and consider the product $\vartheta_p\vartheta_q = \sum_i N_{p,q}^r\vartheta_r$. If $r \not= p+q$, then $N_{p,q}^r = [p-r]N_{q,p-r} + [q-r]N_{p,q-r}$, where $[b] = 0$ if $b < 0$. 
\end{theorem}

\begin{corollary}\label{pstconrel}
$N_{1,q}^r = (q-r)N_{1,q-r}$
\end{corollary}

Using these invariants, we associate to each theta function $\vartheta_p$ the series:
\[N_p = t^p + \sum_i iN_{p,i}t^{-i} \in \mathbb{Q}[[NE[X],t,t^{-1}]].\]
Using these series, we wish to show the following:

\begin{proposition}\label{LGres}
In the notation above, we have for all tuples $(p_1,\ldots,p_d) \in \NN_{> 0}^d$:
\begin{equation}\label{residue}
N_{p_1}\cdots N_{p_d}[t^0] = \vartheta_{p_1}\cdots\vartheta_{p_d}[\vartheta_0]
\end{equation}
\end{proposition}

We will prove this proposition by an argument involving a degeneration of the point constraint, which will facilitate a refinement of the invariant of interest into contributions coming from various decorated tropical types. We first recall a gluing formula derivable from work of Gross in \cite{trglue}.

 \begin{lemma}\label{glem}
 Let $\tau$ be a realizable tropical type of a genus $0$ log stable map to $(X,D)$, $e_1,\ldots e_k,L_1,\ldots,L_m,L_{\out} \in E(G_\tau) \cup L(G_{\tau})$ be the edges and legs containing a vertex $v$, and $L_{\out}$ is a leg of $\tau$ with contact order $0$ with a point constraint $z \in D$ and $\alpha \in H^{\val(v)-3}(\overline{\mathscr{M}}_{0,|L(G_\tau)|})$. Letting $\pmb\tau_i$ be the tropical type produced by cutting at the edge $e_i$ and taking the resulting connected component not containing $v$, and $\pmb\tau_0$ the type produced by cutting at all edges $e_i$ and taking the connected component containing $v$, we have:

\[
 \begin{split}
 \int_{[\mathscr{M}(X,\pmb\tau,z)]^{\vir}} \Forget^*(\alpha) &= \int_{[\mathscr{M}(X,\pmb\tau)]^{\vir}} \Forget^*(\alpha)\ev_{L_{\out}}^*([\pt])\\
 &= m_\tau\int_{[\mathscr{M}(X,\pmb\tau_0)]^{\vir}}\Forget^*(\alpha)\ev_{L_{\out}}^*([\pt])\prod_i \int_{[\mathscr{M}(X,\pmb\tau_i)]^{\vir}} \ev^*_{e_i}(\pt), 
 \end{split}
 \]
 with $m_\tau = |\coker(\prod_i\tau_{i,\NN}^{\gp} \times \prod_i \ZZ \rightarrow \bigoplus_i \pmb\sigma(e_i)) =  |\coker(\ZZ \bigoplus_{i} \ZZ \rightarrow \bigoplus_i \ZZ)|$ the index of the gluing morphism.
 \end{lemma}
 
 \begin{proof}
The first equality follows from the fact that 
\[z_*[\mathscr{M}(X,\pmb\tau,z)]^{\vir} = [\mathscr{M}(X,\pmb\tau)]^{\vir}\cap \ev_{L_{\out}}^*([\pt]),\] for $z: \mathscr{M}(X,\pmb\tau,z) \rightarrow \mathscr{M}(X,\pmb\tau)$ the inclusion map by construction. To prove the second equality, let $\mathscr{M}^{sch}(X,\pmb\tau) = \prod_i \mathscr{M}(X,\pmb\tau_i) \times_{\prod_i X_{\pmb\sigma(e_i)}^2} X_{\pmb\sigma(e_i)}$ be the space of tuples of log curves of type $\pmb\tau_i$ which schematically glue, and $\phi: \mathscr{M}(X,\pmb\tau) \rightarrow \mathscr{M}^{sch}(X,\pmb\tau)$ be the natural morphism. Setting $\Delta: \prod_i X_{\pmb\sigma(e_i)} \rightarrow \prod_i X_{\pmb\sigma(e_i)}^2$ to be the diagonal morphism, then by \cite[Theorem $5.7$]{trglue}, if $\tau$ is tropically transverse in the sense of loc. cit. and $ev: \prod_i\mathfrak{M}^{\ev}(\mathcal{A},\tau_i) \rightarrow \prod_i X_{\pmb\sigma(e_i)}^2$ is flat, then:
 
 \[\phi_*[\mathscr{M}(X,\pmb\tau)]^{\vir} = m_{\tau}\Delta^!(\prod_i [\mathscr{M}(X,\pmb\tau_i)]^{\vir})\]
 Tropical transversality follows from the fact that evaluation morphism $\ev_v: \tau \rightarrow \mathbb{R}_{\ge 0}$ at $v$ is surjective, and flatness of the map $ev$ follows from the fact that the gluing stratum is the deepest stratum.
 
By working in a cohomology theory which has a Kunneth decomposition, we can express the Gysin pullback along the diagonal as pairing the homology class $[\mathscr{M}(X,\pmb\tau_0,z)]^{\vir} \times \prod_{i=1}^k [\mathscr{M}(X,\pmb\tau_i)]^{\vir}$ with $\sum_i \prod_{j=1}^k\gamma_{ij}\otimes \gamma^{ij}$, for $\gamma_{ij}$ a graded basis of the cohomology of $X_{\pmb\sigma(e_i)}$, and $\gamma^{ij}$ the dual basis under the Poincare pairing, which we may assume contains the fundamental class $1$. Since the virtual dimension of $\mathscr{M}(X,\pmb\tau_0,z)$ is $\val(v)-3$, and the insertion $\Forget^*(\alpha)$ of the invariant in question pulls back from the component $\mathscr{M}(X,\pmb\tau_0,z)$ of the product, for a term above to contribute non-trivially, we must have $\gamma_{ij} = 1$ and hence $\gamma^{ij} = [\pt]$. Taking degrees yields the desired equality.

 \end{proof}

 We also require the the following lemma, which will put sufficient constraints on possible contributing tropical types in the upcoming degeneration argument:

\begin{lemma}\label{contcomp}
Let $\pmb\tau$ be a decorated tropical type containing a vertex $v$ with $\pmb\sigma(v) = \mathbb{R}_{>0}$, and suppose $v$ is adjacent to a leg $l$ with contact order $0$. Let $p \in D$ be a point and $\mathscr{M}(X,\pmb\tau,p)$ be the moduli space of log stable maps mapping the marked point $x_l$ associated with $l$ to $p$, then $[\mathscr{M}(X,\pmb\tau,p)]^{\vir} = 0$ if $v$ is decorated by a non-zero curve class.
\end{lemma}

\begin{proof}
By the gluing formula for the virtual class recalled in the proof of Lemma \ref{glem}, we have $[\mathscr{M}(X,\pmb\tau,p)]^{\vir} \not= 0$ only if the tropical type $\pmb\gamma$ produced by cutting edges of $\pmb\tau$ which contains a unique vertex $v$ has non-zero virtual fundamental class. Letting $\textbf{B} \in H_2(X)$ be the degree of $\pmb\gamma$, if $\pmb\gamma$ contains only the leg $l$ as a $1$-dimensional feature, then $0 = \textbf{B}\cdot D$ by the log balancing condition, which implies $\textbf{B} = 0$ since $D$ is ample and $\textbf{B}$ is effective. Thus, we may assume $L(G_{\gamma}) \ge 2$. 

Consider instead the tropical type $\pmb\gamma'$ consisting of a single vertex $v$ of degree $\textbf{B}$ and two legs, one leg $l_{\out}$ with contact order $0$ and another leg with contact order $\textbf{B}\cdot D$. Observe that for any log stable map $f: C \rightarrow X$ over a base $S$ marked by $\pmb\gamma$, we have that $f$ scheme theoretically factors through $D \subset X$. Hence, we can produce a stable map $\overline{f}: C' \rightarrow X$ schematically factoring through $D$ by stabilizing the underlying stable map of $f$ after forgetting the relevant marked points, yielding a partial stabilization morphism $\rho: C \rightarrow C'$. By \cite[Lemma $8.8$]{stacktrop}, after equipping $C$ with a log structure which forgets the relevant collection of marked points, the pushforward log structure gives $C'$ the structure of a log prestable curve. Moreover, the family of tropical curves is marked by curve classes $\NE(X)$ in the sense of \cite[Section $2.2.2$]{punc}. 

To give a log enhancement of the stable map $C \rightarrow X$, we first produce a log enhancement of the prestable map $C' \rightarrow X$. Note that since $f$ maps $C$ entirely into $D$, it suffices to construct a PL function $h: \Sigma(C) \rightarrow \mathbb{R}_{\ge 0}$ such that the corresponding line bundle $\mathcal{O}(h)$ is isomorphic to the pullback of $\mathcal{O}_X(D)$. Note that since $f: C \rightarrow X$ is a log map, we have $f^*\mathcal{O}_X(D) = \mathcal{O}_C(\Sigma(f))$. Moreover, for $\pi: C \rightarrow S$ the family of log curves, we may write $\Sigma(f) = \Sigma(\pi)^*(f_{v}) + f'$ for $f' \in PL(\Sigma(C))$ which vanishes on the subcomplex of cones associated with the leg $l_{\out}$. 

Since $\Sigma(S)$ parameterizes a family of genus zero tropical curves decorated by curve classes in $\NE(X)$, there exists a unique PL function $h'$ on $\Sigma(C)$ which is zero on every cone associated with a vertex contained in $l_{\out}$, and has slopes along edges and legs such that the following log balancing condition is satisfied:
$$\sum_{v \in e} \textbf{u}(e) = C_v\cdot D.$$
Consider now the PL function $h = h' + \Sigma(\pi)^*(f_v)$. Observe now that since $\mathcal{O}(h)$ and $\mathcal{O}(f)$ restricted to each geometric fiber have same degree, and the fibers all have genus zero, $\mathcal{O}(h-f)$ is fiberwise trivial, hence must be pulled back from a bundle on S. Since the restriction of $h - f$ to the subcomplex of cones associated with the leg $l_{\out}$ is zero, the bundle restricted to the section associated with $x_{\out}$ is trivial, and we have $\mathcal{O}(h - f) \cong \mathcal{O}$. If follows that $h: \Sigma(C) \rightarrow \mathbb{R}_{\ge 0}$ now satisfies the desired balancing condition, inducing a prestable log map $C \rightarrow X$.

The log map produced above descends to a log enhancement of the stable map $C' \rightarrow X$ by noting that stabilization removes all bivalent and univalent vertices marked by contracted curve classes, hence $\Sigma(h)$ restricted to such edges are linear, hence $h$ is in the image of the pullback morphism $PL(\Sigma(C')) \rightarrow PL(\Sigma(C))$, which we may also identify with $h$. As $\rho^{-1}: \Pic(C') \rightarrow \Pic(C)$ is injective and $\rho^*\mathcal{O}(h) \cong \rho^*\overline{f}^*\mathcal{O}_X(D)$, it follows that $\mathcal{O}(h) \cong \mathcal{O}_X(D)$. This yields a log map $C' \rightarrow X$ of type $\gamma'$. Taking the unique associated basic log map $C' \rightarrow X$ of type $\gamma'$ gives a morphism $S \rightarrow \mathscr{M}(X,\pmb\gamma')$. 


The above construction on $S$ valued points for all schemes $S$ naturally extends to a morphism of stacks $\mathscr{M}(X,\pmb\gamma)\rightarrow \mathscr{M}(X,\pmb\gamma')$. The moduli stack $\mathscr{M}(X,\pmb\gamma')$ has virtual dimension $dim\text{ }X - 2$ by \cite{punc} Proposition $3.30$ and Riemann-Roch, hence the point constrained moduli stack $\mathscr{M}(X,\pmb\gamma',p)$ has virtual dimension $-1$, and in particular has vanishing virtual fundamental class. Now recall from \cite[Definition $3.8$]{punc} the moduli stacks $\mathfrak{M}(\mathcal{X},\pmb\gamma)$ and $\mathfrak{M}(\mathcal{X},\pmb\gamma')$ of decorated prestable maps to $\mathcal{X}$ marked by the decorated tropical types $\pmb\gamma$ and $\pmb\gamma'$ respectively. The construction of the morphism $\mathscr{M}(X,\pmb\gamma) \rightarrow \mathscr{M}(X,\pmb\gamma')$ also gives a morphism $\mathfrak{M}(\mathcal{X},\pmb\gamma) \rightarrow \mathfrak{M}(\mathcal{X},\pmb\gamma')$, where the partial stabilization is determined by the marking data as in \cite[section $2$]{Cos}. Now observe we have the following commuting square

\[
\begin{tikzcd}
\mathscr{M}(X,\pmb\gamma,p) \arrow{r} \arrow{d} & \mathfrak{M}^{\ev}(\mathcal{A}_{\NN},\pmb\gamma) \times_{D} p \arrow{d}\\
\mathscr{M}(X,\pmb\gamma',p)\arrow{r} & \mathfrak{M}^{\ev}(\mathcal{A}_{\NN},\pmb\gamma')\times_{D} p
\end{tikzcd}
\]

This commuting square is in fact cartesian. To see this, suppose we have morphisms $S \rightarrow \mathscr{M}(X,\pmb\gamma',p)$ and $S \rightarrow \mathfrak{M}(\mathcal{X},\pmb\gamma)$ from a scheme $S$ which induce isomorphic morphisms $S \rightarrow \mathfrak{M}(\mathcal{X},\pmb\gamma')$. Then we have a family of prestable log curves $C \rightarrow S$ with components decorated by cuve classes in $H_2(X)$, and a map $C \rightarrow \mathcal{A}_{\NN}$, as well as a partial stabilization of the family $\overline{C} \rightarrow S$, and a log stable map $\overline{f}: \overline{C} \rightarrow X$ of tropical type $\gamma'$. We lift the underlying stable map $\overline{f}$ to a stable map $f: C \rightarrow X$ by precomposing $\overline{f}$ with the partial stabilization morphism $c: C \rightarrow \overline{C}$. Note now the decoration of $C$ is induced by the stable map $C \rightarrow X$. For this stable map to have a log lift whose image under the forgetful map recovers the given map $C \rightarrow \mathcal{A}_{\NN}$, since $C$ maps entirely into $D$, we only need $\mathcal{O}(\Sigma(f)) = f^*\mathcal{O}(D)$. This follows from the fact that $\mathcal{O}(\Sigma(f)) = c^*\mathcal{O}(\Sigma(\overline{f}))$ by construction of $\mathfrak{M}^{\ev}(\mathcal{X},\pmb\gamma) \rightarrow \mathfrak{M}^{\ev}(\mathcal{X},\pmb\gamma')$, and $f^*\mathcal{O}(D) = c^*\overline{f}^*\mathcal{O}(D) = c^*\mathcal{O}(\Sigma(\overline{f}))$.

The standard obstruction theories on $\mathscr{M}(X,\pmb\gamma)$ and $\mathscr{M}(X,\pmb\gamma')$ also give obstruction theories for the two horizontal morphisms above. Since the morphism of universal families over the moduli spaces of the left vertical arrow commute with the universal morphisms, we find that the obstruction theories are compatible. Hence, by \cite[Lemma $A.12$]{int_mirror}, $[\mathscr{M}(X,\pmb\gamma,p)]^{\vir} = 0$ if $\pmb\gamma$ is has non-zero total degree.
\end{proof}

Using the lemma above, we apply a degeneration argument to show that the descendent log Gromov-Witten invariants $\vartheta_p^k[\vartheta_0]$ are formal residues of the power series $N_p$:

\begin{proof}[Proof of Theorem \ref{LGres}]
Letting $\beta = ((p_1,\ldots,p_d,0),\textbf{B})$ be a tropical type of log stable map, by \cite[Theorem $1.3$]{logWDVV}, $\vartheta_{p_1}\cdots\vartheta_{p_d}[\vartheta_0t^{\textbf{B}}]$ is the log Gromov-Witten invariant:

\begin{equation}\label{pdegin}
N_{p_1,\ldots,p_d}^{\textbf{B}}:= \int_{[\mathscr{M}(X,\beta,z)]^{\vir}}\psi_{x_{\out}}^{d-2}.
\end{equation}

The point $z \in X$ above is a general point of $X$. To refine this invariant, we degenerate to point constraint to be contained in $D$. More precisely, the inclusion $z \rightarrow D$ of the point upgrades to a strict map $z^{\dagger} \rightarrow D$, with $z^\dagger$ the standard log point and $D$ equipped with the pulled back log structure from $X$. Since $\Sigma(X) \cong \mathbb{R}_{\ge 0}$, transversality in the sense of \cite[Definition $2.6$]{int_mirror} is automatic. Thus by the same argument as in \cite[Theorem $7.10$]{int_mirror}, or Steps $3$ and $4$ of the proof of \cite[Theorem $1.1$]{mirrcomp}, the log Gromov-Witten invariant of Equation \ref{pdegin} is equal to to analogous integral over $\mathscr{M}(X,\beta,z)$ for $z \in D$.

In order the relate this invariant with the two point invariants of Definition \ref{2pinv}, we give a virtual decomposition of the moduli stack $\mathscr{M}(X,\beta,z)$, that is, we decompose the fundamental class of $\mathfrak{M}^{\ev}(\mathcal{X},\beta,z)$. 

To that end, let $\pmb\tau$ be a potentially contributing decorated tropical type to the decendent log Gromov-Witten invariant $N_{p_1,\ldots,p_d}^{\textbf{B}}$. By Lemma \ref{contcomp}, the vertex $v$ contained in the leg $L_{\out}$ associated with $x_{\out}$ must satisfy $\pmb\sigma(v) = \mathbb{R}_{>0}$ and is decorated by a contracted curve class. By this fact, together with the log balancing condition, we conclude the sum of the contact orders associated to edges and legs containing $v_0$ must sum to $0$. Moreover, by standard comparison results for $\psi$ classes, we have $\psi_{x_{\out}}$ pulls back from the $\overline{\mathscr{M}}_{0,d+1}$ along the stabilization map $\mathscr{M}(X,\beta,z) \rightarrow \overline{\mathscr{M}}_{0,d+1}$. Since $\psi_{x_{\out}}^{d-2} = [\pt] \in A^{d-2}(\overline{\mathscr{M}}_{0,d+1})$, we must have $v_0$ is $d+1$ valent. Moreover, for $e \in E(G_\tau)$ any edge containing $v_0$, and $G_e$ the graph produced by cutting $G_\tau$ at $e$ and taking the connected component which does not contain $v_0$, $G_e$ must have at most one leg. As in \cite[Theorem $6.4$ Step $2$]{scatt}, the tropical modulus of $\tau$ must be $1$-dimensional, hence must be entirely determined by the image of $v$. In particular, all remaining vertices must map to $0 \in \mathbb{R}_{\ge 0}$. Since $\val(v) = d+1$, any other vertex must be contained in exactly one additional leg, see Figure \ref{ttype}.

Having restricted the tropical behavior of a contributing tropical type, we turn now to computing its contribution. To do so, we first examine the multiplicity of the component $\mathfrak{M}_\eta$ of $\mathfrak{M}^{\ev}(\mathcal{X},\beta,z)$ with $\eta$ a geometric generic point associated with a prestable log map $f:C_\eta \rightarrow \mathcal{X}$ of type $\tau$. The same argument as in \cite[Theorem $6.4$ Step $3$]{scatt} shows that the multiplicity is given by $|\coker(\ev_{v}:\tau_\NN^{\gp} \rightarrow \ZZ)|$, in particular depending only on $\tau$. Letting $e_i \in E(G_{\tau})$ be an enumeration of the compact edges of $G_{\tau}$, and $\textbf{u}(e_i)$ the contact order of $e_i$, the previous multiplicity is thus equal to $\lcm(\textbf{u}(e_1),\ldots,\textbf{u}(e_{|E(G_{\tau})|}))$.  In particular, after noting $Aut(\tau/\beta) = 1$ as any automorphism of $G_{\tau}$ which fixes the legs must be trivial, by taking the virtual pullback of the resulting decomposition of $[\mathfrak{M}^{\ev}(\mathcal{X},\beta,z)]$ into irreducible components, we have:

\[[\mathscr{M}(X,\beta,z)]^{\vir} = \sum_{\pmb\tau} \lcm(\textbf{u}(e_1),\ldots,\textbf{u}(e_{|E(G_{\tau})|})) i_{\tau*}[\mathscr{M}(X,\pmb\tau,z)]^{\vir}\]

%
%
%

For a tropical type $\tau$ described above, we let $N_{\pmb\tau} = deg\text{ } \psi_{x_{\out}}^{d-2}\cap[\mathscr{M}(X,\pmb\tau,z)]^{\vir}$. Note that by Lemma \ref{glem} and $\psi_{x_{\out}}^{d-2}\cap [\overline{\mathscr{M}}_{0,d+1}] = [\pt]$, we have:

\[N_{\pmb\tau} = m_\tau \prod_i N_{p_i,\textbf{u}(e_i)}\]
In the above, $m_\tau = |\coker(\ZZ \bigoplus_{e_i} \ZZ \rightarrow \bigoplus_i \ZZ)|$ with the homomorphism sending the first basis vector to $(-1,\ldots,-1)$, and the basis element corresponding to the length of the edge $e_i$ to $\textbf{u}(e_i)$. Hence, $m_\tau = \frac{\prod_i \textbf{u}(e_i)}{\lcm(\textbf{u}(e_1),\ldots,\textbf{u}(e_{|E(G_{\tau})|})}$. Thus, the invariant \ref{pdegin} is equal to:

\[\sum_{\tau} \prod_{i}  \textbf{u}(e_i)N_{p_i,\textbf{u}(e_i)}\]

Finally, note that the tropical type $\tau$ is uniquely determined by a subset of legs $p_1,\ldots p_t$ which contain the vertex $v_{\out}$ and for every leg $p_i$ not in this list, a positive integer $m_i$ such that $\sum_i m_i = \sum p_i$. Indeed, this data clearly determines a tropical type $\tau$, and the additional condition must hold since the class decorated the vertex $v_{\out}$ is contracted. We can therefore rewrite the above equation as:

\[\sum_{I \subset [d]} \sum_{m_1+\cdots + m_{|[d]\setminus I| = \sum_{i \in I}p_i}}  \prod_{i \in [d]\setminus I} m_i N_{p_i,m_i}\]
The above expression matches the $t^0$ coefficient of $N_{p_1}\cdots N_{p_d}$, as required. 
\end{proof}

\begin{figure}[h]\label{ttype}
\centering
\begin{tikzpicture}
\draw[ball color = red] (0,2) circle (0.1cm);
\draw[ball color = red] (0,1) circle (0.1cm);
\draw[ball color = red] (1,1.5) circle (0.1cm);
\draw[->,black] (0,2)--(4,2);
\draw[->,black] (0,1)--(4,1);
\draw[->,black] (1,1.6)--(4,1.6);
\draw[->,black] (1,1.4)--(4,1.4);
\draw[black] (0,2)--(1,1.5);
\draw[black] (0,1)--(1,1.5);
\draw[ball color = blue] (0,0) circle (0.1cm);

\draw[->,blue] (0,0)--(4,0);

\end{tikzpicture}
\caption{An example of a tropical type $\tau$ contributing to $N_{p_1,\ldots,p_4}^\textbf{B}$.}
\end{figure}

\begin{remark}
We note that Theorem \ref{LGres} follows from \cite[Section $3.3$]{YouPLG}, who argues that assignment of the series $N_p$ gives a ring homomorphism with $N_0 = 1$ using the log/orbifold correspondence of \cite{logroot2}, and the TRR result of Theorem \ref{logTRR}. We prove the result above via a direct degeneration argument as these arguments appear most ready to generalize beyond the smooth divisor setting.
\end{remark}

Before proving Corollary \ref{mcr2}, we note the following proposition, which follows from a straightforward exercise with formal power series:

\begin{lemma}\label{exercise}
Consider the collection of power series of the form $t^k + \sum_{i> 0} a_it^{-i} \in \kk[\NE(X)][[t]]$ for some fixed $k \ge 1$, and suppose we have a known series $N_1 = t + \sum_{i>0} c_it^{-i}$:
\begin{enumerate}
\item If $k = 1$, then for $N,N'$ among this collection, we have $N = N'$ if and only if $N^d[t^0] = N^{'d}[t^0]$ for all $d\ge 0$.
\item For $k\ge 1$ arbitrary and $N,N'$ among this collection, we have $N = N'$ if and only if $NN_1^d[t^0] = N'N_1^d[t^0]$ for all $d\ge 0$. 
\end{enumerate} 
\end{lemma}

\begin{proof}
We prove the first statement, the proof for the second statement is similar. We prove that the coefficient $a_i$ is determined by the periods $N^d[t^0]$ for $d \le i+1$, with case $i = -1$ given by definition. For the induction step, note:
\[N^{i+1}[t^0] = \sum_{d_1+\cdots+d_{i+1} = 0} \prod_{j =1}^{i+1} a_{d_j}.\]
Since the largest positive power of $t$ is $1$, we cannot have $d_j > i$ for any tuple $(d_1,\ldots,d_{i+1})$ contributing to $N^{i+1}[t^0]$. By induction, all terms on the righthand side of the above equation associated to tuples $(d_1,\ldots,d_{i+1})$ with $d_j < i$ for all $j$ are known. The only remaining term is $a_{-1}^{i}a_i = a_i$, hence $a_i$ is determined by $N^d[t^0]$ for $d \le i+1$, as required.
\end{proof}

By Theorem \ref{LGres} and Lemma \ref{exercise}, the log Gromov-Witten invariants $N_{1,i}$ are determined by the values $\vartheta_1^k[\vartheta_0]$ for $k \ge 0$. By Theorem \ref{mthm1}, these values are the regularized quantum periods of $X$. Moerover, using the relation in Corollary \ref{pstconrel}, we may show that the $2$-pointed invariants $N_{1,k}$ determine all structure constants for the products $\vartheta_p\vartheta_1$ for all $p\ge 0$. An induction argument then shows that all remaining structure constants and series $N_p$ are determined, as follows:

\begin{proof}[Proof of Corollary \ref{mcr2}]
The fact that the invariants of $(3)$ from the corollary statement may be used to recover the invariants in $(2)$ follows by Theorem \ref{stconrel}, and Theorem \ref{mthm1} implies that $(2)$ may be used to recover $(1)$. Since the quantum period may be used to recover the series $N_1$ by Theorem \ref{LGres} and Lemma \ref{exercise}, it suffices to compute the series $N_p$ for all $p\in \NN$ from the knowledge of series $N_1$. We proceed by induction on $p \in \NN$ with base case $p=1$. Thus, suppose we know $N_{p}$ for all $p < n$. To compute $N_{n}$, we note:

\[\vartheta_n = \vartheta_1\vartheta_{n-1} - \sum_{j\le n-1} N_{1,n-1}^j\vartheta_j\]
By Theorem \ref{LGres}, we have the following equality of coefficients of power series for any integer $d \ge 0$;

\[N_nN_1^d[t^0] = (N_1N_{n-1} - \sum_{j \le n-1} N_{1,n-1}^jN_j)N_1^d[t^0]\]
We now note that the series $N_n$ and $N_1N_{n-1} - \sum_{j \le n-1} N_{1,n-1}^jN_j$ are of the form $t^n + \sum_{i> 0} a_it^{-i}$. This fact regarding the former series is by definition. For the latter series, we note that 
\[N_1N_{n-1}[t^j] = \begin{cases}
0&j>n\\1& j = n\\ (n-1-j)N_{1,n-1-j}& 1\le j\le n-1
\end{cases}\]
By Corollary \ref{pstconrel}, we have $(n-1-j)N_{1,n-1-j} = N_{1,n-1}^j$. Since $N_j = t^j + \sum_i iN_{j,i}t^{-i}$, the right hand side also has the form $t^n + \sum_{i> 0} a_it^{-i}$. It now follows from Lemma \ref{exercise} that $N_n = N_1N_{n-1} - \sum_{j \le n-1} N_{1,n-1}^jN_j$. Induction now determines the power series $N_p$ for all $p \in \NN$.
\end{proof}

\begin{remark}
We note that the essential place in the proof of Corollary \ref{mcr2} which used the assumption that $X$ is Fano is that the quantum periods are equal to $\vartheta_1^k[\vartheta_0]$ for $k\ge 0$. If we replace knowledge of the quantum periods with knowledge of $\vartheta_1^k[\vartheta_0]$ for $k\ge 0$ in the statement of Corollary \ref{mcr2}, a similar theorem holds more generally.
\end{remark}

\nocite{*}
\bibliographystyle{amsalpha}
\bibliography{regqp+2point}

\end{document}